\documentclass[]{interact}

\usepackage{epstopdf}
\usepackage[caption=false]{subfig}

\usepackage[numbers,sort&compress]{natbib}
\bibpunct[, ]{[}{]}{,}{n}{,}{,}
\renewcommand\bibfont{\fontsize{10}{12}\selectfont}

\usepackage[colorlinks,citecolor=blue,linkcolor=blue]{hyperref}
\usepackage[nameinlink,noabbrev,capitalize]{cleveref}
\usepackage{xcolor}
\usepackage[normalem]{ulem}

\numberwithin{equation}{section}

\theoremstyle{plain}
\newtheorem{theorem}{Theorem}[section]
\newtheorem{lemma}[theorem]{Lemma}
\newtheorem{corollary}[theorem]{Corollary}
\newtheorem{proposition}[theorem]{Proposition}

\theoremstyle{definition}
\newtheorem{definition}[theorem]{Definition}
\newtheorem{example}[theorem]{Example}

\theoremstyle{remark}
\newtheorem{remark}[theorem]{Remark}

\newtheorem{algorithm}[theorem]{Algorithm}

\renewcommand{\subset}{\subseteq}

\renewcommand{\d}{\,\text{\rm{}d}}

\newcommand{\dm}{\,\text{\rm{}d}\mu}
\newcommand{\Om}{\mathrm{\Omega}}
\def\N{\mathbb  N}
\def\R{\mathbb  R}

\def\supp{\operatorname{supp}}

\def\dist{\operatorname{dist}}

\def\conv{\operatorname{conv}}

\DeclareMathOperator{\sign}{sign}
\DeclareMathOperator{\ext}{ext}

\begin{document}

\title{The Largest-$K$-Norm for General Measure Spaces and a DC Reformulation for $L^0$-Constrained Problems in Function Spaces}

\author{
\name{Bastian Dittrich\textsuperscript{a}, Daniel Wachsmuth\textsuperscript{a}
\thanks{CONTACT {\tt bastian.dittrich@uni-wuerzburg.de, daniel.wachsmuth@uni-wuerzburg.de}}
}
\affil{\textsuperscript{a}Institut f\"ur Mathematik, Universit\"at W\"urzburg, 97074 W\"urzburg, Germany}
}

\maketitle

\begin{abstract}
We consider constraints on the measure of the support for integrable functions on arbitrary measure spaces.
It is shown that this non-convex and discontinuous constraint can be equivalently reformulated by the difference of two convex and continuous functions, namely the $L^1$-norm and the so-called largest-$K$-norm.
The largest-$K$-norm is studied and its convex subdifferential is derived.
A corresponding penalty method is proposed, and its numerical solution by a DC method is investigated.
Numerical experiments for two example problems, including a sparse optimal control problem, are presented.
\end{abstract}


\begin{amscode}
49K20, 49M20
\end{amscode}

\section{Introduction}

We are interested in solving optimal control problems with constraints on the measure of the support of the control.
These constraints are formulated in terms of the $L^0$-pseudo-norm, which is defined by
\[
 \|u\|_0 := \mu( \supp (u) ) = \mu \left( \{x\in \Omega \mid u(x) \ne 0\} \right),
\]
where $u: \Omega \to \R$ is measurable, and $(\Omega,\mathcal A, \mu)$ is a measure space.
The resulting optimal control problem can be written as the minimization problem
\begin{equation} \label{eq_intro_L0constrained}
\min_{u\in X} f(u)  \quad \text{subject to} \quad \|u\|_0 \le K,
\end{equation}
where $X$ is a Banach space of measurable functions from $\Omega$ to $\R$,  and $f: X \to \overline{\R}$ is convex and lower semicontinuous.

In the case of optimal control problems for partial differential equations,
$\mu$ is the Lebesgue measure on $\R^d$, and the underlying function space is usually chosen as $X= L^p(\Omega)$, $p \in (1,\infty)$.
Such problems were studied for instance in \cite{ItoKunisch2014,KaliseKunischSturm2018,Munch2008,Wachsmuth2019,Wachsmuth2023}.
Due to the lack of weak lower semicontinuity of $u\mapsto \|u\|_0$ in $L^p(\Omega)$-spaces,
standard existence proofs do not work. A simple example without solution is presented in \cite[Section 4.5]{Wachsmuth2019}.
In \cite{ItoKunisch2014,LouZhao2024} existence of solutions is shown under restrictive assumptions on the problem data. Nevertheless,
optimality conditions in terms of Pontryagin's maximum principle can be proven \cite{ItoKunisch2014,Wachsmuth2023}.
In addition, algorithms based on topological derivatives are available \cite{KaliseKunischSturm2018,Munch2008,Wachsmuthtop}.

Motivated by the difficulties that arise when the problem is set in $L^p(\Om)$-spaces,
we focus on optimization problems, where $X$ is compactly embedded into $L^1(\Omega)$.
Due to the compact embedding, we can prove existence of solutions, see \cref{TheoremMinGlobalSolution}.
However, neither the  Pontryagin's maximum principle nor the above mentioned optimization methods are applicable
in this setting.

In order to solve problem \eqref{eq_intro_L0constrained} in this setting numerically, we follow the ideas of \cite{Gotoh2017}, where optimization problems in $\R^n$ with cardinality constraints are studied.
They proposed a penalization scheme based on the largest-$K$-norm, which in the case of a general measure space $(\Omega,\mathcal A, \mu)$ can be defined as
\[
\vert u \vert_K := \sup_{A \in \mathcal{A}: \: \mu(A) \leq K} \int_A \vert u \vert \d\mu.
\]
As observed in \cite{Gotoh2017}, one has the equivalence
\begin{equation}\label{eq_intro_equivalence}
\Vert u \Vert_0 \leq K \quad \Leftrightarrow \quad \Vert u \Vert_1 - \vert u \vert_K = 0,
\end{equation}
which is true for every $u \in L^1(\Om)$ and every $\sigma$-finite measure space $(\Omega,\mathcal A, \mu)$, see \cref{Theo:ReformSepa}.
This justifies to solve a sequence of penalized problems
\begin{equation}\label{eq_intro_penalizedproblem}
\min_{u \in X} f(u) + \rho \left( \Vert u \Vert_1 - \vert u \vert_K \right), \quad \rho > 0,
\end{equation}
where $\rho$ is a penalization parameter.
The penalized problems \eqref{eq_intro_penalizedproblem} have a particular structure: the objective functional is a difference of convex functions (DC function),
which enables the use of specialized optimization methods \cite{Thi2018,Thi2014b}.

In the resulting DC-Algorithm, a subgradient of the largest-$K$-norm has to be computed in each step.
Hence, we investigate the subdifferential of $u \mapsto \vert u \vert_K$ in detail in \cref{Section:LargestKNorm} below.
A full characterization of the subdifferential in an atom-free measure space, e.g., Lebesgue measure space, can be found in \cref{theo:Subdiff}.
For numerical implementation, we need to discretize the problem. The discretized
problems are set in a purely atomic measure space. Hence, we also study this situation in \cref{sec_general_measure_space}.
Surprisingly, the proofs and computations are much more technical, and we only have a partial result about the subdifferential of the largest-$K$-norm.
The results for not necessarily atom-free measure spaces can also be applied to sparse signal recovery
problems, as for instance investigated in \cite{Blumensath2008,Wang2013,Grasmair2010,Ramlau2012,Ito2013}.

Convergence results for the DC-Algorithm in general Banach spaces are shown in \cref{Section:DCConvergence}, including an inexact convergence result.
\cref{Section:PenalizedProblem} is then devoted to the study of the $L^0$-constrained problem and its penalization.
We show existence of solutions for the original as well as the penalized problem and prove a classical penalty method result.
We investigate optimality conditions of the penalized problem and generalize the classical penalty method result to sequences of critical points instead of global minima.
The goal of \cref{Section:ExampleProblem} is to numerically solve a specific $L^0$-constrained problem.
Reformulating and penalizing this specific problem follows the ideas of the previous sections based on the largest-$K$-norm.
We report about the numerical results in \cref{section:NumExperiments}
and also consider an optimal control problem in \cref{Section:ControlProblem}.

\section*{Notation}

Throughout the whole paper we assume that  $(\Om, \mathcal{A},\mu)$ is a $\sigma$-finite measure space.
Hence, there are pairwise disjoint sets $(\Om_n)$ with $\mu(\Om_n) < \infty$ and $\Om = \bigcup_{n=1}^\infty \Om_n$.

With the notation
\[
\vert x \vert_0 := \begin{cases}
1 & \text{ if } x \neq 0, \\
0 & \text{ if } x = 0, \\
          \end{cases} \qquad x \in \R,
\]
we can write for any measurable $u: \Om \to \R$ the $L^0$-pseudo-norm as
\[
\Vert u \Vert_0 = \int_\Om \vert u(x) \vert_0 \dm(x).
\]

Let $X$ be a normed space with topological dual $X'$.
If $u \in X$ and $s \in X'$ we denote the dual pairing by $\left\langle s, u \right\rangle_{X',X} := s(u)$.
If the spaces $X$ and $X'$ are clear from the context we simply write $\langle \cdot , \cdot \rangle$.
We write $u_n \overset{X}{\to} u$, if $(u_n) \subseteq X$ converges to $u \in X$ with respect to the norm in $X$ or $u_n \overset{X}{\rightharpoonup} u$, if $(u_n) \subseteq X$ converges weakly to $u \in X$.
If the space $X$ is clear from the context we simply write $u_n \to u$ and $u_n \rightharpoonup u$ respectively.
Analogously, we write $s_n \rightharpoonup^* s$ if $(s_n) \subseteq X'$ converges weakly-* to $s \in X'$.

Let $f: X \to \overline{\R}$ be a function.
We call $f$ proper if $f(u) > - \infty$ for all $u \in X$ and if there exists some $u \in X$ with $f(u) < \infty$.
We say that $f$ is lower semicontinuous if $u_n \to u$ implies $f(u) \leq \liminf_{n \to \infty} f(u_n)$ for every $(u_n) \subseteq X$ and $u \in X$.
Analogously, we call $f$ weakly lower semicontinuous if $u_n \rightharpoonup u$ implies $f(u) \leq \liminf_{n \to \infty} f(u_n)$ for every $(u_n) \subseteq X$ and $u \in X$.
Moreover, we call $f$ coercive if $\Vert u_n \Vert_X \to \infty$ implies $f(u_n) \to \infty$ for every sequence $(u_n) \subseteq X$ .
Let $Y \subseteq X$ be a normed space.
If $f$ is convex on $Y$ we denote the convex subdifferential of $f$ at $u \in Y$ with respect to the space $Y$ as $\partial^Y f(u)$, i.e.,
\[
\partial^Y f(u) := \left\lbrace s \in Y' \mid f(v) - f(u) \geq \langle s, v - u \rangle_{Y',Y} \ \forall v \in Y \right\rbrace.
\]
If the space $Y$ is clear from the context or if $Y = X$ we omit the superscript $Y$ in the notation.

\section{Preliminary Results}\label{Section:Preliminary}

\subsection{DC Reformulation in Finite Dimensions}

First, let us recall the definition of DC functions.

\begin{definition}
Let $X$ be a vector space, $f: X \to \overline{\R}$  a mapping. Then $f$ is called DC function (difference-of-convex function)
if there are convex functions $f_1:X \to \bar\R$  and $f_2:X\to \R$ such that $f = f_1-f_2$.
\end{definition}

Let us briefly review the developments of \cite{Gotoh2017}.
There, the cardinality-constrained problem
\begin{align} \label{FiniteProblem}
\min_{x \in \R^n} f(x) \quad \text{subject to} \quad \Vert x \Vert_0 \leq K,\: x \in S
\end{align}
is studied, where
\[
\Vert x \Vert_0 := \left\vert \{ i \in \{1,...,n\} \mid x_i \neq 0 \} \right\vert
\]
is the cardinality of the support of $x \in \R^n$,
$K \in \{1,...,n\}$ is given,
$f: \R^n \to \R$ is a DC function, and $S \subset \R^n$ is a closed convex set.
Based on the largest-$K$-norm
\[
\vert x \vert_K := \vert x_{(1)} \vert + \hdots + \vert x_{(K)} \vert = \max_{I \subseteq \{1,...,n\}: \vert I \vert \leq K} \sum_{i \in I} \vert x_i \vert ,
\]
where $x_{(i)}$ is the $i$-th largest entry of $x$ in absolute value, they consider the problem
\[
\min_{x \in \R^n} f(x) \quad \text{subject to} \quad \Vert x \Vert_1 - \vert x \vert_K = 0,\: x \in S,
\]
which is equivalent to \eqref{FiniteProblem}
due to the following theorem.

\begin{theorem} (\cite[Theorem 1]{Gotoh2017})
For every $x \in \R^n$ and arbitrary integers $K, h  $ with $1 \leq K < h \leq n$ the following statements are equivalent:
\begin{enumerate}
\item $\Vert x \Vert_0 \leq K$,
\item $\vert x \vert_h - \vert x \vert_K = 0$,
\item $\Vert x \Vert_1 - \vert x \vert_K = 0$.
\end{enumerate}
\end{theorem}
Based on this result, the associated penalized problem
\begin{align}\label{FiniteProblemPen}
\min_{x \in S} f(x) + \rho \left( \Vert x \Vert_1 - \vert x \vert_K \right), \quad \rho > 0,
\end{align}
is used as an approximation.
Since $x \mapsto \vert x \vert_K$ is convex, the objective functional in \eqref{FiniteProblemPen}
is  a DC function.
Then the penalized problems can be solved by a DC-Algorithm in order to solve the original problem \eqref{FiniteProblem}.
For the convergence of the penalty method, they have the following result.

\begin{theorem} (\cite[Theorem 2]{Gotoh2017})
Let $0 < \rho_1 < \rho_2 < ...$ with $\lim_{t \to \infty} \rho_t = \infty$ and suppose that $x^t \in \R^n$ is a global minimizer of the penalized problem \eqref{FiniteProblemPen} with penalty parameter $\rho = \rho_t$. Then every limit point $x^*$ of the sequence $(x^t)$ is a global solution of the cardinality-constrained problem \eqref{FiniteProblem}.
\end{theorem}

Moreover, in the case of $f$ being continuously differentiable with Lipschitz-continuous gradient and, e.g., $S = \R^n$, they provide an exact penalty result \cite[see][Theorem 3]{Gotoh2017}.
We will show a similar result in \cref{cor_exact_penalty}.

\subsection{Atom-Free Measure Spaces}

Let us recall the definition and basic properties of atom-free measure spaces.

\begin{definition}
A set $A \in \mathcal{A}$ with $\mu(A) > 0$ is called atom, if every set $B \in \mathcal{A}$ with $B \subset A$ either satisfies $\mu(B) = 0$ or $\mu(A \setminus B)=0$.
The measure space $(\Om, \mathcal{A}, \mu)$ is called atom-free if there are no atoms.
The measure space $(\Om, \mathcal{A}, \mu)$ is called purely atomic if for every $B \in \mathcal{A}$ with $\mu(B) > 0$ there exists an atom $A \in \mathcal{A}$ with $A \subset B$.
\end{definition}

These definitions are consistent with \cite[211I-K]{Fremlin2003}.
We will frequently use the following results.

\begin{theorem}[{\cite[215D]{Fremlin2003}}]\label{Theorematomfree}
Let $(\Om, \mathcal{A}, \mu)$ be atom-free.
Then for every $A \in \mathcal{A}$ with $\mu(A) < \infty$ and $t \in [0, \mu(A)]$ there exists a set $B \in \mathcal{A}$ with $B \subseteq A$ and $\mu(B)=t$.
\end{theorem}

\begin{corollary}\label{Coratomfree}
Let $(\Om, \mathcal{A}, \mu)$ be an atom-free and $\sigma$-finite measure space.
Then for every $A \in \mathcal{A}$ and $t \in [0, \mu(A)]$ there exists a set $B \in \mathcal{A}$ with $B \subseteq A$ and $\mu(B)=t$.
\end{corollary}

\begin{proof}
We only have to consider the case $t < \mu(A) = \infty$.
Let $\Om = \bigcup_{n=1}^\infty \Om_n$ be a decomposition of $\Om$ in pairwise disjoint sets $\Om_n \in \mathcal{A}$ with $0 < \mu(\Om_n) < \infty$.
Since $t < \infty$, there exists $N \in \N$ such that $t \leq \mu\left(\bigcup_{n=1}^N (\Om_n \cap A)\right) < \infty$.
Applying \cref{Theorematomfree} to the set $\bigcup_{n=1}^N (\Om_n \cap A)$ then shows the claim.
\end{proof}

\section{The Largest-$K$-Norm} \label{Section:LargestKNorm}

In this section, we will investigate basic properties of the largest-$K$-norm.

\begin{definition}[largest-$K$-norm] \label{DefkNorm}
For every $u \in L^1(\Om)$ and $K \in [0, \mu(\Om)]$ we define
\[
\vert u \vert_K := \sup_{A \in \mathcal{A}: \: \mu(A) \leq K} \int_A \vert u \vert \d\mu.
\]
\end{definition}

Clearly, $\vert\cdot\vert_K \equiv 0$ for $K = 0$ and  $\vert\cdot\vert_K = \Vert\cdot\Vert_1$ for $K=\mu(\Om)$.
In addition, $K \mapsto \vert u \vert_K$ is monotonically increasing for fixed  $u \in L^1(\Om)$.
As $\vert \cdot \vert_K$ is the supremum of convex and lower semicontinuous functions, it is convex and lower semicontinuous.

The connection to the $L^1(\Om)$-norm directly implies continuity.
\begin{lemma}\label{Lem:Cont}
Let $K \in [0, \mu(\Om)]$.
Then the function $u \mapsto \vert u \vert_K$ is a continuous semi-norm on $L^1(\Om)$.
\end{lemma}

\begin{proof}
It is easy to verify that $\vert \cdot \vert_K$ satisfies the triangle inequality and  is absolutely homogeneous.
Let $(u_n) \subseteq L^1(\Om)$ and $u \in L^1(\Om)$ with $\Vert u_n - u \Vert_1 \to 0$.
Then
\[
\vert \vert u_n \vert_K - \vert u \vert_K \vert \leq \vert u_n - u \vert_K \leq \Vert u_n - u \Vert_1,
\]
which shows the claim.
\end{proof}

The main goal of this section is to show that the supremum in \cref{DefkNorm} is attained, which is surprisingly difficult.
First, we state a sufficient condition that a set $A$ realizes the supremum.
As it turns out, upper level sets play an important role for the largest-$K$-norm.

\begin{lemma} \label{lem_omega_t}
Let $u \in L^1(\Om)$. Define for every $t \in [0, \infty)$ the sets
\[
\Om_{>t} := \{x \in \Om \mid \vert u(x) \vert > t \}, \quad
\Om_{\geq t} := \{x \in \Om \mid \vert u(x) \vert \geq t \}
.
\]
Then the function $[0, \infty) \ni t \mapsto \mu(\Om_{>t})$ is monotonically decreasing and continuous from the right, and the function $[0, \infty) \ni t \mapsto \mu(\Om_{\geq t})$ is monotonically decreasing and continuous from the left.
\end{lemma}
\begin{proof}
First, note that for any $s,t \in [0, \infty)$ with $s < t$ it holds $\Om_{>t} \subset \Om_{>s}$ and $\Om_{\geq t} \subset \Om_{\geq s}$.
Let $t \in [0, \infty)$, $s \in (0, \infty)$. Take sequences $(s_n)$ and $(t_n)$ with $t_n \searrow t$ and $s_n \nearrow s$.
W.l.o.g. we can assume that the sequences $(t_n)$ and $(s_n)$ are monotone and $s_1 > 0$.
Then the claim follows directly from the continuity of the measure from above and below since
\begin{align*}
\Om_{>t} = \bigcup_{n \in \N} \Om_{>t_n}, \qquad \Om_{\geq s} = \bigcap_{n \in \N} \Om_{\geq s_n},
\end{align*}
where we have used that $ \mu( \Om_{\geq s_1} ) \le s_1^{-1} \|u\|_1 < \infty$.
\end{proof}

Now we are in the position to prove a sufficient condition that a given set $A$ realizes the supremum in the definition of the largest-$K$-norm.

\begin{lemma}\label{lem_goal1_suff}
Let $u \in L^1(\Om)$ and $K \in (0, \mu(\Om))$ be given.
Suppose there exist $t\ge0$ and $A\in \mathcal A$ such that $\Om_{>t} \subset A \subset \Om_{\ge t}$ and $\mu(A)=K$.
Then we have $\vert u \vert_K = \int_{A} \vert u \vert \d\mu$.
\end{lemma}
\begin{proof}

Let $t\ge0$ and $A\in \mathcal A$ such that $\Om_{>t} \subset A \subset \Om_{\ge t}$ and $\mu(A)=K$.
Take $B\in \mathcal A$ with $\mu(B) \leq K$.
Then we have
\[
\begin{split}
\int_A \vert u \vert \dm - \int_B \vert u \vert \dm
&= \int_{A \setminus B} \vert u \vert \dm - \int_{B \setminus A} \vert u \vert \dm \\
&\ge t( \mu( A \setminus B) - \mu(B\setminus A) ) \\
&= t( \mu(A) - \mu(A\cap B) - ( \mu(B) - \mu(A\cap B))) \\
&= t( K - \mu(B)) \ge0.
\end{split}
\]
This proves $\int_A \vert u \vert \dm = \vert u \vert_K$.
\end{proof}

We start our investigation of the largest-$K$-norm with the special case of an atom-free  measure space,
which covers the case of the Lebesgue measure space.
Later on, we will
consider the general situation, which includes
finite-dimensional problems and problems set in sequence spaces,
where the underlying measure is a weighted counting measure.
Most of the following results hold trivially for $K \in \{0, \mu(\Om)\}$. Hence, we only consider $K \in (0, \mu(\Om))$.

\subsection{The Atom-Free Case}\label{section:atom-free}

For atom-free measure spaces, $u \mapsto \vert u \vert_K$ is a norm.

\begin{theorem}\label{Theo:Norm}
Let $(\Om, \mathcal{A}, \mu)$ be atom-free. For every $K \in (0, \mu(\Om))$ the function $\vert \cdot \vert_K: L^1(\Om) \to \R$ is a norm on the space $L^1(\Om)$.
\end{theorem}

\begin{proof}
By \cref{Lem:Cont} $\vert \cdot \vert_K$ is already a semi-norm.
Let $K \in (0, \mu(\Om))$ and take $u \neq 0$ in $L^1(\Om)$. Then there exists a set $A \in \mathcal{A}$ with $\mu(A) > 0$ and $u(x) \neq 0$ a.e.\@ on $A$.
Due to \cref{Coratomfree} we can choose $B\subset A$ such that $0 < \mu(B) \leq K$.
It follows $\vert u \vert_K \geq \int_B \vert u \vert \dm > 0$.
This shows that $\vert \cdot \vert_K$ is a norm on $L^1(\Om)$.
\end{proof}

We will now use the upper level sets from \cref{lem_omega_t} to show that a set satisfying the sufficient condition of \cref{lem_goal1_suff} exists, and hence
the supremum in the definition of the largest-$K$-norm is always attained.

\begin{theorem} \label{Goal1}
Let $(\Om, \mathcal{A}, \mu)$ be atom-free.
Let $u \in L^1(\Om)$ and $K \in (0, \mu(\Om))$ be given. Then
there are $t\ge0$ and $A\in \mathcal A$ such that $\Om_{>t} \subset A \subset \Om_{\ge t}$, $\mu(A)=K$, and
$\vert u \vert_K = \int_{A} \vert u \vert \d\mu$.
Here, $\Om_{>t}$ and $\Om_{\ge t}$ are defined as in \cref{lem_omega_t}.
\end{theorem}
\begin{proof}
First, let us observe that $\mu( \Omega_{\ge0} ) = \mu(\Omega)$ and $\lim_{t\to\infty} \mu( \Omega_{\ge t} )=0$.
Since $s \mapsto \mu( \Omega_{\ge s} )$ is monotonically decreasing and continuous from the left by \cref{lem_omega_t}, $t := \sup\{ s\ge0 \mid \mu(\Omega_{\ge s})\ge K\}$ is well-defined and it holds
\[
\mu\left(\Omega_{\ge t+\varepsilon}\right) < K \le \mu\left(\Omega_{\ge t}\right) \quad \forall \varepsilon >0.
\]
Since  $\mu\left(\Omega_{>t}\right)= \mu\left(\bigcup_{k=1}^\infty \Omega_{\ge t+\frac1k}\right)=\lim_{n\to\infty} \mu\left(\Omega_{\ge t+\frac1n}\right)$, it follows
\[
\mu\left(\Omega_{>t}\right) \le K \le \mu\left(\Omega_{\ge t}\right) .
\]
Due to \cref{Coratomfree}, we can choose $A \in \mathcal{A}$ such that $\mu(A) = K$ and $\Omega_{>t} \subset A \subset \Omega_{\ge t}$.
Then the claim follows with \cref{lem_goal1_suff}.
\end{proof}

Note that the set $A$ from \cref{Goal1} is not unique as soon as the set $\Om_{=t} := \Om_{\geq t} \setminus \Om_{>t}$ satisfies $\mu(\Om_{=t}) > 0$.
\cref{Goal1} enables us to prove the key result for the reformulation of a $L^0$-constrained problem by the largest-$K$-norm.

\begin{theorem} \label{GenTheorem1}
Let $(\Om, \mathcal{A}, \mu)$ be atom-free.
Let $u \in L^1(\Om)$ and $K \in (0, \mu(\Om))$ be given. Then the following four statements are equivalent:
\begin{enumerate}
\item \label{GenTheorem1_a} $\Vert u \Vert_0 \leq K$,
\item \label{GenTheorem1_d} $\Vert u \Vert_1 = \vert u \vert_K $,
\item \label{GenTheorem1_c} $\vert u \vert_{H}= \vert u \vert_K $ for all $H \in (K, \mu(\Om)]$,
\item \label{GenTheorem1_b} $\vert u \vert_h = \vert u \vert_K $ for one  $h \in (K, \mu(\Om)]$.
\end{enumerate}
\end{theorem}

\begin{proof}
(\ref{GenTheorem1_a}) $\Rightarrow$ (\ref{GenTheorem1_d}):
Let $A:=\supp(u)$. Then $\mu(A)\le K$ and $|u|_K \ge \int_A |u|\d\mu  =\|u\|_1 \ge |u|_K $.
(\ref{GenTheorem1_d}) $\Rightarrow$ (\ref{GenTheorem1_c}): This follows from the monotonicity of $H \mapsto \vert u \vert_{H}$ and $\|u\|_1 = |u|_{\mu(\Om)}$.
(\ref{GenTheorem1_c}) $\Rightarrow$ (\ref{GenTheorem1_b}) is trivial.

(\ref{GenTheorem1_b}) $\Rightarrow$ (\ref{GenTheorem1_a}):
Let $A$ be given by \cref{Goal1}, i.e.,
$\mu(A) \le K$ and $\int_A |u|\d\mu = |u|_K$.
By replacing $A$ with $A\cap \supp u$, we can assume $A \subset\supp u$.
Let $B$ be given such that $A \subset B \subset \supp u$ and $\mu(B) \le h$.
Then $|u|_h \ge \int_B |u|\d\mu \ge \int_A |u|\d\mu  = |u|_K = |u|_h$,
which implies $u=0$ a.e.\@ on $B\setminus A$.
Due to the inclusion $B \subset \supp u$ it follows $\mu(B\setminus A)=0$ and $\mu(B)=\mu(A)\le K$.
Since $\mu$ is atom-free,
all subsets of $\supp u$ have measure $\le K$, and $\|u\|_0 \le K$ follows.
\end{proof}

\begin{remark}\label{Remark:Reform}
In \cref{GenTheorem1} only the implication (\ref{GenTheorem1_b}) $\Rightarrow$ (\ref{GenTheorem1_a}) needed that the measure space is atom-free.
If there are atoms then this implication is not true: take $\Om=\{0\}$ with $\mu(\Om)=1$.
Let $u$ be given with $u(0)=1$. Then $\|u\|_0=1$, $|u|_h=0$ for all $h<1$, so that \cref{GenTheorem1} (\ref{GenTheorem1_b}) is satisfied for $K=\frac12$ but (\ref{GenTheorem1_a}) is not.
\end{remark}

Let us now investigate the subdifferential of $u \mapsto |u|_K$.
The first step is to prove that the supremum in the definition of $|\cdot|_K$ can be written as the
supremum of a linear functional on a convex set.
We start with  the following auxiliary result.

\begin{lemma}\label{LemmaExtD}
Let $(\Om, \mathcal{A}, \mu)$ be atom-free.
Let $K \in (0, \mu(\Om))$ be given and define
\[
D := \left\lbrace d \in L^{\infty}(\Om) \mid d \geq 0,\, \Vert d \Vert_{\infty} \leq 1,\, \Vert d \Vert_1 \leq K \right\rbrace.
\]
Then $D$ is convex and it holds
\[
\ext(D) = \left\lbrace \chi_A \in L^{\infty}(\Om) \mid A\in \mathcal{A},\,  \mu(A) \leq K \right\rbrace,
\]
where $\ext(D)$ denotes the set of extreme points of $D$.
\end{lemma}

\begin{proof}
Obviously, $D$ is convex.
Let $A \in \mathcal{A}$ with $\mu(A) \leq K$, which implies $\chi_A \in D$.
Let $d_1, d_2 \in D$ and $\lambda \in (0,1)$ be given such that $\lambda d_1 + (1-\lambda)d_2 = \chi_A$.
Then for almost all $x \in \Om$ we have $d_1(x),d_2(x) \in [0,1]$ and $\chi_A(x) \in \ext( [0,1] ) $.
It follows $\chi_A = d_1 = d_2$ almost everywhere, and hence $\chi_A$ is an extreme point of $D$.

Now let $d \in D \setminus  \{ \chi_A \mid A\in \mathcal{A},\,  \mu(A) \leq K \}$.
Then $d$ cannot be a characteristic function.
Therefore, there exist a set $A \in \mathcal{A}$ with $0 < \mu(A) < \infty$ and $\varepsilon > 0$ such that $\varepsilon < d < 1-\varepsilon$ a.e.\@ on $A$.
Since $(\Om, \mathcal{A}, \mu)$ is atom-free there exists a set $B \subseteq A$ with $\mu(B) = \frac{1}{2}\mu(A)$.
We define
\[
d_1 := d + \varepsilon (\chi_B - \chi_{A\setminus B}), \quad d_2 := d + \varepsilon (\chi_{A\setminus B} - \chi_B).
\]
Then it holds $d_1, d_2 \in D$, $\frac{1}{2} d_1 + \frac{1}{2} d_2 = d$, and $d_1 \neq d \neq d_2$, which proves $d \notin \ext(D)$.
\end{proof}

The supremum in \cref{DefkNorm} of the largest-$K$-norm can be taken over the whole set $D$ instead of just the characteristic functions without changing the result.

\begin{theorem} \label{TheoremAltKRepr}
Let $(\Om, \mathcal{A}, \mu)$ be atom-free.
Let $u \in L^1(\Om)$ and $K \in (0, \mu(\Om))$ be given. Let $D$ be as in \cref{LemmaExtD}. Then it holds
\[
\sup_{A \in \mathcal{A}: \mu(A) \leq K} \int_A \vert u \vert \dm = \vert u \vert_K = \sup_{d \in D} \int_{\Om} d \vert u \vert \dm.
\]
Both suprema are attained.
Moreover, if $\Vert u \Vert_0 \geq K$ and $d \in D$ attains the supremum then it holds $\Vert d \Vert_1 = K$.
\end{theorem}
\begin{proof}
Let $t$ and $A$ be as in \cref{Goal1},
i.e., $\Om_{>t} \subset A \subset \Om_{\ge t}$ with $\mu(A)=K$ and $|u|_K=\int_A |u|\d\mu$.
Define the sets $\Om_{<t}$ and $\Om_{=t}$ analogously to \cref{lem_omega_t}.

Let $d\in D$ be given. Using the properties of $d$ and $A$, we find
\begin{multline} \label{eq_sup_D}
\int_\Om d |u| \dm - |u|_K = \int_\Om (d - \chi_A)|u| \dm \\
\begin{aligned}
&=
\int_{ \Om_{>t} } (d-1)|u| \dm + \int_{ \Om_{=t} } (d-\chi_A)|u| \dm + \int_{ \Om_{<t} } (d-0) |u| \dm \\
&\le
t \left( \int_{ \Om_{>t} } (d-1) \dm + \int_{ \Om_{=t} } (d-\chi_A)  \dm + \int_{ \Om_{<t} } d \dm  \right)\\
&=
t ( \|d\|_1  - \mu(A) ) \le0.
\end{aligned}\end{multline}
This proves
\[
 \vert u \vert_K \ge \sup_{d \in D} \int_{\Om} d \vert u \vert \dm.
\]
Since for $A \in \mathcal{A}$ with $\mu(A) \leq K$ we have $\chi_A\in D$, the reverse inequality follows.

Let $\Vert u \Vert_0 \geq K$.
Let the supremum be attained in $d\in D$.
The inequalities in \eqref{eq_sup_D} are satisfied with equality if and only if
$d=1$ a.e.\@ on $\Om_{>t}$, $d=0$ a.e.\@ on $\Om_{<t}$, and $t ( \|d\|_1  - K ) =0$.
If $t>0$ then we get $ \|d\|_1  = K$ immediately.
If $t=0$ then these conditions imply $\|d\|_1 \ge \mu( \Om_{>0}) = \Vert u \Vert_0 \geq K$.
Since $d\in D$, it follows $ \|d\|_1  = K$.
\end{proof}

The proof implies that all sets which attain the supremum in the largest-$K$-norm satisfy, up to null sets, the sufficient condition from \cref{lem_goal1_suff}.
\begin{corollary}
Let $(\Om, \mathcal{A}, \mu)$ be atom-free.
Let $u \in L^1(\Om)$, $K \in (0, \mu(\Om))$ and $t \geq 0$ from \cref{Goal1}.
If $\Vert u \Vert_0 \geq K$ and $B$ is a set with $\mu(B) \leq K$ and $\int_B \vert u \vert \dm = \vert u \vert_K$ then $\mu(B) = K$ and $\mu(\Om_{>t} \setminus B) = 0 = \mu(B \setminus \Om_{\geq t})$.
\end{corollary}

\begin{proof}
\cref{TheoremAltKRepr} directly implies $\mu(B) = K$ and the corresponding proof shows $\chi_B=1$ a.e.\@ on $\Om_{>t}$ and $\chi_B=0$ a.e.\@ on $\Om_{<t}$ which is the claim.
\end{proof}

Moreover, the characterization by \cref{TheoremAltKRepr} has the following two consequences.
First, in the case of a finite measure space, we have that the largest-$K$-norm is an equivalent norm on $L^1(\Om)$.

\begin{corollary}\label{CorNormEquiv}
Let  $(\Om, \mathcal{A}, \mu)$ be finite and atom-free.
Let $K \in (0, \mu(\Om))$. Then it holds
\[
\frac{K}{\mu(\Om)} \Vert u \Vert_1 \leq \vert u \vert_K \leq \Vert u \Vert_1 \quad \forall u \in L^1(\Om),
\]
i.e., $\vert \cdot \vert_K$ and $\Vert \cdot \Vert_1$ are equivalent norms for $L^1(\Om)$.

\end{corollary}

\begin{proof}
By \cref{Theo:Norm} $\vert \cdot \vert_K$ is a norm on $L^1(\Om)$.
Define $d := \frac{K}{\mu(\Om)} \chi_{\Om} \in D$. We then have for every $u \in L^1(\Om)$ that
\[
\frac{K}{\mu(\Om)} \Vert u \Vert_1 = \int_{\Om} d \vert u \vert \dm \leq \sup_{d \in D} \int_{\Om} d \vert u \vert \dm = \vert u \vert_K.
\]
\end{proof}

Note that setting $u = \chi_{\Om}$ and $u = \chi_A$ for some $A \in \mathcal{A}$ with $\mu(A) = K$ shows that these two inequalities are sharp.
A second consequence of \cref{TheoremAltKRepr}
is the following generalization of the H\"older inequality.

\begin{corollary}\label{cor:Integral}
Let $(\Om, \mathcal{A}, \mu)$ be atom-free.
Let $u\in L^1(\Omega)$, $v\in L^1(\Omega) \cap L^\infty(\Omega)$ and $K \in (0, \mu(\Om))$. Then it holds
\[
\int_\Omega u \cdot v \dm \le |u|_K \cdot \max\left( \|v\|_\infty, K^{-1} \|v\|_1\right).
\]
\end{corollary}

\begin{proof}
This follows directly from \cref{TheoremAltKRepr}, since for $v\ne0$ we have
\[
\left( \max \left( \|v\|_\infty, K^{-1} \|v\|_1) \right) \right)^{-1} \vert v \vert \in D.
\]
\end{proof}

Using \cref{TheoremAltKRepr}, we can fully characterize the convex subdifferential of the largest-$K$-norm.
We need the following well-known fact, see, e.g.,  \cite[Proposition 16.18]{BauschkeCombettes2011} for a similar result.
We present its short proof for the convenience of the reader.

\begin{lemma}\label{lem:Subdiff}
Let $X$ be a normed space and $f: X \to \R$ be convex and positively homogeneous, i.e.,
$f(\lambda x) = \lambda f(x)$ for all $x\in X$ and all $\lambda\ge0$.
Then for every $x \in X$ it holds
\[
\partial f(x) = \left\lbrace s \in \partial f(0) \mid \langle s, x \rangle = f(x) \right\rbrace.
\]
\end{lemma}

\begin{proof}
First, note that $f$ satisfies $f(0)=0$.
Let $s \in \partial f(0)$ with $\langle s, x \rangle = f(x)$, then
\[
f(y) - f(x) \geq \langle s, y \rangle - \langle s, x \rangle = \langle s, y - x \rangle
\]
for every $y \in X$, i.e., $s \in \partial f(x)$.
Now let $s \in \partial f(x)$, i.e.,
\[
f(y) - f(x) \geq \langle s, y - x \rangle \quad \forall y \in X.
\]
Setting $y = 0$ and $y = 2x$ implies $f(x) \leq \langle s, x \rangle$ and $f(x) \geq \langle s, x \rangle$, respectively, which shows $\langle s, x \rangle = f(x)$.
Inserting this in the above inequality yields
\[
f(y) \geq \langle s, y \rangle \quad \forall y \in X,
\]
i.e., $s \in \partial f(0)$.
\end{proof}

\begin{theorem}\label{theo:Subdiff}
Let $(\Om, \mathcal{A}, \mu)$ be atom-free.
Let $0 \neq u \in L^1(\Om)$ and $K \in (0, \mu(\Om))$. Then it holds
\begin{align*}
\partial \vert \cdot \vert_K(0) &= \left\lbrace s \in L^{\infty}(\Om) \mid \Vert s \Vert_{\infty} \leq 1,\, \Vert s \Vert_1 \leq K \right\rbrace, \\
\partial \vert \cdot \vert_K(u) &= \left\lbrace s \in L^{\infty}(\Om) \mid \max\left( \Vert s \Vert_{\infty},\, K^{-1}\Vert s \Vert_1 \right) = 1,\, \int_{\Om} s u \dm = \vert u \vert_K \right\rbrace.
\end{align*}
\end{theorem}
\begin{proof}
Since $u \mapsto  \vert u \vert_K$ is convex and positively homogeneous, it follows
\[
\partial \vert \cdot \vert_K(u) = \left\lbrace s\in \partial \vert \cdot \vert_K(0) \mid \ \int_{\Om} s u \dm = \vert u \vert_K \right\rbrace
\]
by \cref{lem:Subdiff}. Therefore, it suffices to compute $\partial \vert \cdot \vert_K(0)$.

Let $s\in L^{\infty}(\Om)$ be given with $\Vert s \Vert_{\infty} \leq 1$ and $\Vert s \Vert_1 \leq K$.
Then \cref{cor:Integral} implies $\int_\Om s \cdot v \dm \leq \vert v \vert_K$ for all $v \in L^1(\Omega)$, i.e., $s\in \partial |\cdot|_K(0)$.
Now let $s \in \partial |\cdot|_K(0)$.
This is equivalent to
\[
\int_\Omega s \cdot v \dm \le |v|_K
\]
for all $v \in L^1(\Omega)$. Suppose $\Vert s \Vert_\infty > 1$, take $A \subset \{x \in \Omega \mid \vert s(x) \vert >1\}$ such that $\mu(A) \in (0,K)$ and set $v = \chi_A \sign(s)$. Then $\int_\Omega s \cdot v \dm > \mu(A) = |v|_K$.
Hence $\|s\|_\infty \le 1$.

Let $v_n := \sum_{k=1}^n \sign(s) \chi_{\Om_k} \in L^1(\Om)$, where $(\Om_k)$ is defined at the beginning of \cref{section:atom-free}.
Then the sequence $(sv_n)$ is non-negative and converges pointwise a.e.\@ to $\vert s \vert$.
Fatou's Lemma implies
\[
\Vert s \Vert_1 = \int_\Om \vert s \vert \dm \leq \liminf_{n \to \infty} \int_\Om s v_n \dm \leq \liminf_{n \to \infty} \vert v_n \vert_K \leq K.
\]
Let $s \in \partial \vert \cdot \vert_K(u)$.
Then $s \in L^1(\Om) \cap L^\infty(\Om)$ with $\int_\Om su \dm = \vert u \vert_K$ and
\cref{cor:Integral} implies that $\max\left( \Vert s \Vert_{\infty},\, K^{-1}\Vert s \Vert_1 \right) = 1$, which shows the claim.
\end{proof}

\begin{remark}
Let $K \in (0, \mu(\Om))$.
If $(\Om, \mathcal{A}, \mu)$ is atom-free, the largest-$K$-norm is a norm by \cref{Theo:Norm}.
Using  \cref{theo:Subdiff} and \cref{cor:Integral}, one can easily verify
that the dual space of $(L^1(\Om),\vert \cdot \vert_K)$ is given by $L^1(\Om) \cap L^\infty(\Om)$.
Furthermore, the dual norm, defined by
\[
\vert s \vert_K' := \sup_{u \in L^1(\Om): \vert u \vert_K \leq 1} \left\vert \langle s, u \rangle \right\vert,
\]
can be characterized as
\[
\vert s \vert_K' = \max\left( \Vert s \Vert_\infty, K^{-1} \Vert s \Vert_1 \right)
\quad \forall s \in (L^1(\Om),\vert \cdot \vert_K)'.
\]
In the finite dimensional setting this was already proven in \cite{Watson1992}.
\end{remark}

\subsection{The General Case}
\label{sec_general_measure_space}

In this section we now investigate the case where $(\Om, \mathcal{A}, \mu)$ is a general, not necessarily atom-free, $\sigma$-finite measure space.
Recall that we can write
$\Om = \bigcup_{n=1}^\infty \Om_n$ with $(\Om_n)$ pairwise disjoint and $\mu(\Om_n) < \infty$.
Due to \cite{Johnson1970}, we can decompose the measure $\mu$ into an atom-free and a purely atomic part.

\begin{proposition}\label{prop_decomposition}
There are measures $\mu_{\mathrm{pa}}$ and $\mu_{\mathrm{af}}$  such that $\mu = \mu_{\mathrm{pa}} + \mu_{\mathrm{af}}$, where $\mu_{\mathrm{pa}}$ is purely atomic, and $\mu_{\mathrm{af}}$ is atom-free.
These measures are mutually singular ($\mu_{\mathrm{pa}}  \perp \mu_{\mathrm{af}}$), i.e.,
there exists $Y \in \mathcal{A}$ such that $\mu_{\mathrm{pa}} (A) = \mu(A \cap Y)$ and $\mu_{\mathrm{af}}(A) = \mu(A \setminus Y)$ for all $A\in \mathcal A$.
In addition, $\mu_{\mathrm{af}}(A) = 0$ for all atoms $A \in \mathcal A$ of $\mu$.
\end{proposition}

\begin{proof}
The decomposition was constructed in \cite[Corollary 2.6]{Johnson1970} in a more general setting, while the mutual singularity is a consequence of $\sigma$-finiteness and \cite[Theorem 3.3]{Johnson1967}.
Let $A \in \mathcal A$ with $\mu_{\mathrm{af}}(A) > 0$. Then there is $B \subset A$ with $\mu_{\mathrm{af}}(B)\in (0,\mu_{\mathrm{af}}(A))$.
It follows $\mu_{\mathrm{pa}}(B) \in [0,\mu_{\mathrm{pa}}(A)]$ and $\mu(B) \in (0,\mu(A))$, and $A$ is not an atom of $\mu$.
\end{proof}

The goal of this section is to generalize the results of the atom-free case from \cref{section:atom-free}.
In contrast to \cref{Theo:Norm}, $\vert \cdot \vert_K$ is not a norm in this general setting.
If there is an atom $X \in \mathcal A$ of $\mu$ with $\mu(X)>K$, then $\vert \chi_{X} \vert_K = 0$.
Nevertheless, we will prove the existence of a set that attains the supremum in the definition of $\vert \cdot \vert_K$.
In addition, we will prove that the reformulation of the $L^0$-constraint is applicable in the general case.

In addition, we also give some counter examples to results from \cref{section:atom-free},  which are no longer valid. These counterexamples will be based on the following setup.

\begin{example}\label{Ex:Setup}
We choose
\[
\Om := (0,1) \cup (1,3) \cup (3,6), \quad \mathcal{A} := \mathcal{A}_\sigma\left( \mathcal{L}((0,1)) \cup \{(1,3), (3,6) \} \right),
\]
so that $\mathcal{A}$ is the $\sigma$-algebra generated by $ \{(1,3), (3,6) \} $ and the Lebesgue measurable subsets of $(0,1)$.
In addition, we take $\mu$ to be the Lebesgue measure.
The resulting measure space is separable and finite, but neither atom-free nor purely atomic.
Define the function
\[
u := 4 \left( \chi_{(0,1)} + \chi_{(1,3)}  \right) + 3 \chi_{(3,6)}.
\]
For $K := 4$ we then obviously have
\[
\vert u \vert_K = \int_{(0,1) \cup (3,6)} \vert u \vert \dm = 13.
\]
Note that for all the following counter examples a purely atomic measure space generated by three atoms, e.g., $\mathcal{A}_\sigma\left(\{(0,1), (1,3), (3,6) \} \right)$, would be sufficient.
Nevertheless, we use this example to demonstrate that even if the measure space contains an atom-free part,
some results from \cref{section:atom-free} are not true in the general setting.
\qed
\end{example}

Since the measure space is $\sigma$-finite, we have $L^\infty(\Om) = L^1(\Om)'$.
The set $D$ from \cref{LemmaExtD} is compact in the weak-* topology on $L^\infty(\Om)$.

\begin{lemma}\label{lem_D_compact}
Let $K \in (0, \mu(\Om))$ be given.
Then the set
\[
D = \left\lbrace d \in L^{\infty}(\Om) \mid d \geq 0,\, \Vert d \Vert_{\infty} \leq 1,\, \Vert d \Vert_1 \leq K \right\rbrace
\]
from \cref{LemmaExtD} is weak-* compact in $L^\infty(\Om)$.
\end{lemma}

\begin{proof}
We can write
\[
D = \{ d \geq 0\} \cap \{ \Vert d \Vert_{\infty} \leq 1 \} \cap \{ \Vert d \Vert_1 \leq K \} =: D_1 \cap D_2 \cap D_3.
\]
The closed unit ball $D_2$ is known to be weak-* compact by Banach-Alaoglu \cite[Theorem 3.16]{Brezis2011}.
Therefore, it suffices to show that $D$ is weak-* closed.

For any $v \in L^1(\Om)$ define the function $f_v: L^\infty(\Om) \to \R$ by $f_v(d) := \int_\Om d v \dm$.
Then, since $f_v$ is weak-* continuous for every $v \in L^1(\Om)$, and
\[
D_1 = \bigcap_{v \in L^1(\Om): v \geq 0} f_v^{-1}([0,\infty)), \quad D_3 = \bigcap_{v \in L^\infty(\Om): \Vert v \Vert_\infty \leq 1} f_v^{-1}([-K, K]),
\]
this shows that $D_1$ and $D_3$, and therefore $D$, is weak-* closed.
\end{proof}

This result enables us to prove the existence of a set that attains the supremum in the largest-$K$-norm.

\begin{theorem}\label{Theo:GeneralExist}
Let $u \in L^1(\Om)$ and $K \in (0, \mu(\Om))$ be given.
Then there exists a set $A \in \mathcal{A}$ with $\mu(A)\leq K$ and $\vert u \vert_K = \int_{A} \vert u \vert \dm$.
\end{theorem}

\begin{proof}
Let
\[
C := \left\lbrace \chi_A \in L^\infty(\Om) \mid A \in \mathcal{A}, \, \mu(A) \leq K \right\rbrace.
\]
Let $\overline{C}^*$ denote the weak-* closure of $C$ in $L^\infty(\Om)$.
Due to $C \subset D$, the set $\overline{C}^*$ is weak-* compact by \cref{lem_D_compact}.
Since the functional $f \mapsto \int_\Om f \vert u \vert \dm$ is weak-* continuous, we have
\[
\vert u \vert_K = \sup_{f \in C} \int_\Om f \vert u \vert \dm = \sup_{f \in \overline{C}^*} \int_\Om f \vert u \vert \dm,
\]
and hence there exists $f \in \overline{C}^* \subset D$, which attains the supremum.

Let $X \subset  \Omega$ be an atom of $\mu$. Define the weak-* continuous functional $g: L^\infty(\Om) \to \R$ by $g(f) := \int_X f \dm$.
For any $\chi_A \in C$ it holds $g(\chi_A) = \mu(A \cap X) \in \{0, \mu(X)\}$ and hence the continuity of $g$ implies
\[
g\left( \overline{C}^* \right) \subseteq \overline{g(C)} = \{0, \mu(X)\}.
\]
It follows $\int_X f \dm \in \{0, \mu(X)\}$, so that $f|_X = 0$ or $f|_X = 1$ $\mu$-a.e.\@ on $X$.
Hence, there is $B \subset Y$, with $Y$ from \cref{prop_decomposition}, such that $f= \chi_B$ $\mu_{\mathrm{pa}}$-almost everywhere. Due to \cref{prop_decomposition}, we have $\mu_{\mathrm{af}}(B)=0$.

Since the measure $\mu_{\mathrm{af}}$ is atom-free, we can apply \cref{Goal1} and \cref{TheoremAltKRepr}.
Let us set $K' :=  \int_\Om f \dm_{\mathrm{af}} \in [0,K]$.
Then there is $A \in \mathcal A$ with $\mu_{\mathrm{af}}(A) \le K'$ such that
\[
 \int_A \vert u \vert \dm_{\mathrm{af}} \ge \int_\Om g \vert u \vert \dm_{\mathrm{af}}
\]
for all $g\in L^\infty(\Omega)$ with $g \geq 0$, $\|g\|_\infty \leq 1$ and $\int_\Om g \dm_{\mathrm{af}} \le K'$.
Due to \cref{prop_decomposition}, we can assume $\mu_{\mathrm{pa}}(A)=0$.
This implies $\mu( A \cup B) \le K$ and $\mu(A\cap B)=0$.
In addition, we get
\begin{align*}
\int_{A \cup B} |u| \dm &= \int_A |u|\dm_{\mathrm{af}} + \int_B |u| \dm_{\mathrm{pa}} \\
&\ge \int_\Om f |u|\dm_{\mathrm{af}}+ \int_\Om f |u| \dm_{\mathrm{pa}} = \int_\Om f|u|\dm = \vert u \vert_K,
\end{align*}
which is the claim.
\end{proof}

The following example shows that there is not always a set which has both, or even only one, of the sufficient properties of \cref{lem_goal1_suff} and attains the supremum in $\vert \cdot \vert_K$.

\begin{example}
In the setup of \cref{Ex:Setup} we have
\begin{align*}
\mu(\Om_{>t}) &= \mu(\Om) = 6 > K \quad \text{for all } t < 3, \\
\mu(\Om_{\geq t}) &= \mu((0,1) \cup (1,3)) = 3 < K \quad \text{for all } 3 < t \leq 4, \\
\Om_{>3} &= (0,1) \cup (1,3), \\
\Om_{\geq 3} &= \Om.
\end{align*}
Therefore, there can only exist a set $\Om_{>t} \subseteq A \subseteq \Om_{\geq t}$ with $\mu(A) = K$ for $t = 3$.
In this case no such set $A$ exists, since the only possible sets are $(0,1) \cup (1,3)$ and $\Om$.

Moreover, the supremum is not even attained by a feasible upper level set since
\[
\int_{\Om_{>3}} \vert u \vert \dm = 12 < 13 = \vert u \vert_K.
\]
Furthermore, if we choose $K = 4.5$ the supremum is still exclusively attained by $(0,1) \cup (3,6)$, i.e., for a set $A$ with $\mu(A) < K$.
\hfill \qed
\end{example}

As already mentioned in \cref{Remark:Reform}, the proof for the equivalent reformulations of the $L^0$-constraint only required the measure space to be atom-free in one step.
Therefore, the reformulation as $\Vert u \Vert_1 - \vert u \vert_K = 0$ is still possible in the general case.

\begin{theorem}\label{Theo:ReformSepa}
Let $u \in L^1(\Om)$ and $K \in (0, \mu(\Om))$ be given.
Then the following statements are equivalent:
\begin{enumerate}
\item \label{Reform_a} $\Vert u \Vert_0 \leq K$,
\item \label{Reform_b} $\Vert u \Vert_1 = \vert u \vert_K $,
\item \label{Reform_c} $\vert u \vert_{H}= \vert u \vert_K $ for all $H \in (K, \mu(\Om)]$.
\end{enumerate}
In addition, these statements imply
\begin{enumerate}
\setcounter{enumi}3
\item \label{Reform_d} $\vert u \vert_h = \vert u \vert_K $ for one $h \in (K, \mu(\Om)]$.
\end{enumerate}
\end{theorem}

\begin{proof}
The implications (\ref{Reform_a}) $\Rightarrow$ (\ref{Reform_b})  $\Rightarrow$ (\ref{Reform_c})  $\Rightarrow$ (\ref{Reform_d}) follow
as in the proof of \cref{GenTheorem1}.
Since (\ref{Reform_c}) trivially implies (\ref{Reform_b}), we only have to show (\ref{Reform_b}) $\Rightarrow$ (\ref{Reform_a}):
By \cref{Theo:GeneralExist} there exists $A \in \mathcal{A}$ with $\mu(A) \leq K$ and $\int_A \vert u \vert \dm = \vert u \vert_K$.
By considering $A \cap \supp(u)$, we can assume that $A \subseteq \supp(u)$.
Then $\Vert u \Vert_1 = \vert u \vert_K$ implies
\[
0 = \int_\Om \vert u \vert \dm - \int_A \vert u \vert \dm = \int_{\supp(u) \setminus A} \vert u \vert \dm,
\]
and therefore $\mu(\supp(u) \setminus A) = 0$.
This shows $\Vert u \Vert_0 = \mu(A) \leq K$.
\end{proof}

\begin{example}
Note that implication (\ref{Reform_d}) $\Rightarrow$ (\ref{Reform_a}) of  \cref{Theo:ReformSepa} does not hold
in the setup of \cref{Ex:Setup}: Here, we have $\Vert u \Vert_0 = 6 > K$ but $\vert u \vert_K = 13 = \vert u \vert_h$ for every $K < h < 5$.
\end{example}

In contrast to \cref{LemmaExtD},
the set $D$ has extreme points that are not characteristic functions.
This suggests that $D$ is much larger than
the weak-* closed and convex hull of $C= \left\lbrace \chi_A \in L^\infty(\Om) \mid A \in \mathcal{A}, \, \mu(A) \leq K \right\rbrace$.

\begin{lemma}
Let $K \in (0, \mu(\Om))$ be given.
Then the extreme points $\ext(D)$ of the set $D$ from \cref{lem_D_compact} are given by
\begin{multline*}
\ext(D) =
\left\lbrace \chi_A \mid \mu(A) \leq K \right\rbrace
\\
\cup \ \left\lbrace \chi_A + \alpha \chi_{X} \mid X \text{ atom} ,\, A \cap X = \emptyset, \, \alpha \in (0,1), \, \mu(A) + \alpha \mu(X) = K \right\rbrace.
\end{multline*}
\end{lemma}

\begin{proof}
Let $d$ be in the set on the right. Take $d_1,d_2 \in D$, $\lambda \in (0,1)$ with $d = \lambda d_1 + (1-\lambda)d_2$.
If $d = \chi_A$ with $\mu(A) \leq K$ we have $d_1(x), d_2(x) \in [0,1]$ and $d(x) \in \ext([0,1])$ for almost all $x \in \Om$, and therefore $d_1 = d = d_2$ almost everywhere.
This shows $d \in \ext(D)$.

If $d = \chi_A + \alpha \chi_X$ we get as before $d_1 = d = d_2$ on $\Om \setminus X$.
For $j=1,2$ we therefore have
\[
\int_{X} d_j \dm = \Vert d_j \Vert_1 - \int_{\Om \setminus X} d \dm \leq K - \mu(A) = \alpha \mu(X),
\]
and hence $0 \leq d_j \leq \alpha$ a.e.\@ on $X$.
With the same reasoning as before we get $d_1 = \alpha = d = d_2$ a.e.\@ on $X$ and therefore $d \in \ext(D)$.

Now let $d \in D$ be not in the set on the right.
Then $d \in D \setminus \{ \chi_A \mid \mu(A) \leq K \}$, and therefore $d$ cannot be a characteristic function.
Hence, we have $\mu(E) > 0$ for $E := \{ x \in \Om \mid 0 < d(x) < 1\}$.
We distinguish three cases:
(1) $\mu_{\mathrm{af}} (E) > 0$, (2) $E$ is an atom, (3) $E$ contains at least two disjoint atoms.

(1) There exists $\tilde{E} \subseteq E$ and $\varepsilon > 0$ such that $0 < \mu_{\mathrm{af}}(\tilde{E}) < \infty$, $\varepsilon < d < 1-\varepsilon$ a.e.\@ on $\tilde{E}$, and $\mu_{\mathrm{pa}}(\tilde{E})=0$.
Since $\mu_{\mathrm{af}}$ is atom-free, there exists $B \subseteq \tilde{E}$ with $\mu(B) = \mu_{\mathrm{af}}(B) = \frac12 \mu_{\mathrm{af}}(\tilde{E}) = \frac{1}{2} \mu(\tilde{E})$.
Let us define the functions
\[
d_1 := d + \varepsilon \left( \chi_B - \chi_{\tilde{E} \setminus B} \right),
\quad
d_2 := d + \varepsilon \left( \chi_{\tilde{E} \setminus B} - \chi_B \right).
\]
Then it follows $d_1,d_2\in D$, $d_1 \neq d \neq d_2$, $d = \frac12(d_1+d_2)$, and $d\not\in\ext(D)$.

(2)
Since $d$ is measurable, and $E$ is an atom, there exists $\alpha\in (0,1)$ such that $d=\alpha$ a.e.\@ on $E$.
Then there is $A \in \mathcal{A}$ such that $d = \chi_A + \alpha \chi_E$ with $A \cap E = \emptyset$.
Due to the properties of $d$, it follows $\mu(A) + \alpha \mu(E) < K$.
Hence, there exists $\varepsilon > 0$ such that
\[
\mu(A) + (\alpha + \varepsilon) \mu(E) \leq K, \quad 0 \leq \alpha - \varepsilon, \quad \alpha + \varepsilon \leq 1.
\]
Define the functions
\[
d_1 := d + \varepsilon \chi_{E},
\quad
d_2 := d - \varepsilon \chi_{E}.
\]
Then it follows $d_1,d_2\in D$, $d_1 \neq d \neq d_2$, $d = \frac12(d_1+d_2)$, and $d\not\in\ext(D)$.

(3) In this case there exist atoms $X_1, X_2$ with $X_1 \cap X_2 = \emptyset$ and $\alpha_1, \alpha_2 \in (0,1)$ with
\[
d = d\vert_{\Om \setminus (X_1 \cup X_2)} + \alpha_1 \chi_{X_1} + \alpha_2 \chi_{X_2}.
\]
Then there exists $\varepsilon > 0$ such that
\[
(\alpha_1 \pm \varepsilon) \in [0,1], \quad (\alpha_2 \pm \varepsilon) \in [0,1].
\]
Assume w.l.o.g. that $\mu(X_1) \leq \mu(X_2)$.
Therewith we define the functions
\begin{align*}
d_1 &:= d\vert_{\Om \setminus (X_1 \cup X_2)} + (\alpha_1 - \varepsilon) \chi_{X_1} + \left(\alpha_2 + \varepsilon \frac{\mu(X_1)}{\mu(X_2)}\right) \chi_{X_2}, \\
d_2 &:= d\vert_{\Om \setminus (X_1 \cup X_2)} + (\alpha_1 + \varepsilon) \chi_{X_1} + \left(\alpha_2 - \varepsilon \frac{\mu(X_1)}{\mu(X_2)}\right) \chi_{X_2}.
\end{align*}
Again, we get $d_1,d_2 \in D$, $d_1 \neq d \neq d_2$, and $d = \frac{1}{2}(d_1 + d_2)$.
Therefore, this shows $d \notin \ext(D)$.
\end{proof}

The alternative supremum over the set $D$ from \cref{TheoremAltKRepr} is still attained but in general larger than $|u|_K$.

\begin{theorem}
For every $u \in L^1(\Om)$ and $K \in (0, \mu(\Om))$ it holds
\[
\sup_{d \in D} \int_{\Om} d \vert u \vert \dm \geq \vert u \vert_K,
\]
where $D$ is from \cref{lem_D_compact}.
The supremum on the left-hand side is attained.
\end{theorem}
\begin{proof}
This is a direct consequence of \cref{lem_D_compact} and $\{\chi_A \mid \mu(A) \leq K \} \subseteq D$.
\end{proof}

The following example shows that we do not have equality in general, which is  in contrast to \cref{TheoremAltKRepr}.

\begin{example}\label{Ex:SupremDiff}
For the setup of \cref{Ex:Setup} define
\[
s := \chi_{(0,1)} + \chi_{(1,3)} + \frac{1}{3} \chi_{(3,6)} \in D.
\]
Then we get
\[
\sup_{d \in D} \int_\Om d \vert u \vert \dm \geq \int_\Om s \vert u \vert \dm = 15 > 13 = \vert u \vert_K.
\]
This also shows that the inequality of  \cref{cor:Integral} does not hold in general, when the measure space has atoms.

Furthermore, there hence exists $s \in L^\infty(\Om)$ with $\Vert s \Vert_\infty \leq 1$ and $\Vert s \Vert_1 \leq K$, as well as $u \in L^1(\Om)$, such that $\int_\Om s \cdot u \dm > \vert u \vert_K$, i.e., $s \notin \partial \vert \cdot \vert_K(0)$.
This shows that \cref{theo:Subdiff} does in general not hold anymore.
\qed
\end{example}

Using the arguments of \cref{Theo:GeneralExist}, we have the following characterization of the subdifferential of the largest-$K$-norm.
Because of \cref{lem:Subdiff} it suffices to study $\partial \vert \cdot \vert_K(0)$.

\begin{theorem}\label{Theo:GeneralSubdiffExact}
Let $K \in (0, \mu(\Om))$ be given. Then it holds
\[
 \partial \vert \cdot \vert_K(0) = \overline{\conv}^* \big( \left\{ f \in L^\infty(\Om) \mid \vert f \vert = \chi_A, \ A \in \mathcal A, \ \mu(A) \le K \right\} \big),
\]
where $\overline{\conv}^*$ denotes the weak-* closed convex hull in $L^\infty(\Omega)$.
\end{theorem}
\begin{proof}
As argued  in the proof of \cref{Theo:GeneralExist}, we have
\[
\vert u \vert_K = \sup_{f \in C} \int_\Om f \vert u \vert \dm = \sup_{\vert f \vert \in C} \int_\Om f u \dm = \sup_{f \in \overline{\conv}^*(\{\vert f \vert \in C\})} \int_\Om f u \dm,
\]
where $C = \{ \chi_A \mid A \in \mathcal A, \ \mu(A) \le K\}$.
This directly implies $\overline{\conv}^*(\{\vert f \vert \in C\}) \subset \partial \vert \cdot \vert_K(0)$.

Let us take $s \not\in \overline{\conv}^*(\{\vert f \vert \in C\})$. Then by a consequence of Hahn-Banach theorem \cite[Theorem 3.4]{Rudin1991} (in the weak-* topology of $L^\infty(\Om)$) there is, by \cite[Korollar VIII.3.4]{Werner2018}, $v \in L^1(\Om)$ such that
\[
\int_\Om s v \dm > \sup_{f \in \overline{\conv}^*(\{\vert f \vert \in C\})} \int_\Om f v \dm = \vert v \vert_K,
\]
i.e., $s \notin \partial \vert \cdot \vert_K(0)$.
\end{proof}

It seems to be challenging to get a more explicit description of the subdifferential on measure spaces with atoms.
Already in case of the largest-$K$-norm on $\R^n$ this is very technical, see \cite{Watson1992,Wu2014}.
We can still show the following inclusions for the subdifferential of the largest-$K$-norm.

\begin{theorem}\label{Theo:GeneralSubdiff}
Let $K \in (0, \mu(\Om))$ be given. Then it holds
\begin{multline*}
\left\lbrace s \in L^{\infty}(\Om) \mid \Vert s \Vert_{\infty} \leq 1,\ \Vert s \Vert_0 \leq K \right\rbrace
\\
\subseteq
\partial \vert \cdot \vert_K(0)
\subseteq
\\
\left\lbrace s \in L^{\infty}(\Om) \mid \Vert s \Vert_{\infty} \leq 1,\ \Vert s \Vert_1 \leq K \right\rbrace .
\end{multline*}
\end{theorem}
\begin{proof}
It holds $\{f \in L^\infty(\Om) \mid \vert f \vert \in C\} \subset \{s \in L^\infty(\Om) \mid \| s \|_\infty \leq 1, \ \| s \|_1 \leq K\}$, where the latter is convex and weak-* compact by \cref{lem_D_compact}.
\cref{Theo:GeneralSubdiffExact} then implies the second inclusion.

Now let $s \in L^\infty(\Om)$ with $\Vert s \Vert_\infty \leq 1$ and $\Vert s \Vert_0 \leq K$.
Then $\mu(\supp(s)) \leq K$ and we get
\[
\int_\Om s \cdot v \dm = \int_{\supp(s)} s \cdot v \dm \leq \int_{\supp(s)} \vert v \vert \dm \leq \vert v \vert_K
\]
for all $v \in L^1(\Om)$, i.e., $s \in \partial \vert \cdot \vert_K(0)$.
\end{proof}

Both inclusions in \cref{Theo:GeneralSubdiff} are not always satisfied with equality, see, e.g., \cref{Ex:SupremDiff}.

\section{DC-Algorithm in Banach Spaces}\label{Section:DCConvergence}

The penalization of the constraint $\|u\|_1 - |u|_K=0$ leads to a DC function.
The minimization of such functions can be attempted using the celebrated DC-Algorithm.
For the convenience of the reader, we provide some important results on the DC-Algorithm in a general Banach space setting.
For further reading on (mostly finite dimensional) DC problems and the DC-Algorithm we refer to \cite{Thi2018} and references therein.
Convergence properties of the DC-Algorithm in Hilbert spaces were proven by \cite{Lim2023}.

Let us consider the general DC problem
\[
\min_{u \in X} f(u) := g(u) - h(u),
\]
where
\begin{itemize}
\item $X$ is a Banach space,
\item $g:X \to \bar\R$ is proper, convex and lower semi-continuous,
\item $h: X \to \R$ is convex and continuous.
\end{itemize}
The basic DC-Algorithm for this problem reads:

\begin{algorithm}[DC-Algorithm]\label{GenDCA}
\begin{itemize}
\item[]
\item[\text{(S.0)}] Choose $u^0 \in X$ and set $k := 0$.
\item[\text{(S.1)}] Choose $s^k \in \partial h(u^k)$ and determine $u^{k+1}$ as a solution to
\begin{equation}
\min_{u \in X} g(u) - h(u^k) - \langle s^k, u - u^k \rangle.
\label{generalDC-subproblem}
\end{equation}
\item[\text{(S.2)}] If a suitable termination criterion is satisfied: STOP.
\item[\text{(S.3)}] Set $k \leftarrow k+1$ and go to (S.1).
\end{itemize}
\end{algorithm}

Therefore, in step (S.1) the function $h$ is approximated by its linearization based on the subgradient $s^k$, and the resulting convex problem \eqref{generalDC-subproblem} is solved.
Note that the convex function $u\mapsto g(u) - h(u^k) - \langle s^k, u - u^k \rangle$ majorizes $f$.
The following necessary optimality conditions can be derived analogously to the finite dimensional setting.

\begin{proposition}[{\cite[Theorem 2]{Tao1997}}] \label{prop:DCnecessary}
Let $u \in X$ be a local minimum of $f=g-h$. Then
\[
\emptyset \neq \partial h(u) \subseteq \partial g(u),
\]
and in particular $\partial g(u) \cap \partial h(u) \neq \emptyset$.
\end{proposition}

In general, limit points of iterates of \cref{GenDCA} will only satisfy the weaker optimality condition $\partial g(u) \cap \partial h(u) \neq \emptyset$.

\begin{definition}
A point $u \in X$ is called a critical point of $f = g-h$, if
\[
\partial g(u) \cap \partial h(u) \neq \emptyset.
\]
$u$ is called a strongly critical point of $f = g-h$, if
\[
\emptyset \neq \partial h(u) \subseteq \partial g(u).
\]
\end{definition}
A strongly critical point is also often called directional-stationary or d-stationary, see, e.g., \cite{Oliveira2018}.
The following convergence results for \cref{GenDCA} are well-known in finite dimensions, see, e.g., \cite[Theorem 3]{Tao1997}, and directly translate to our Banach space setting.

\begin{lemma}\label{Lem:DCconvergence}
Let $(u^k)$ be a sequence generated by \cref{GenDCA} for $f = g-h$. Then
\begin{enumerate}
\item the sequence $(f(u^k))$ of function values is monotonically decreasing.
\item if $u^{k+1} = u^k$, the current iterate $u^k$ is a critical point of $f = g-h$.
\item if $f(u^k) = f(u^{k+1})$, the current iterate $u^k$ is a critical point of $f = g-h$.
If in addition $g$ or $h$ is strongly convex, it holds $u^{k+1} = u^k$.
\end{enumerate}
\end{lemma}

Under suitable assumptions, accumulation points of the iterates of the DC-Algorithm are critical points of $f = g-h$.
Note that the assumption of the existence of  strong and weak-* convergent subsequences is quite strong.
We prove that they are fulfilled for some problem classes in \cref{Section:ExampleProblem}.

\begin{theorem}\label{Theo:DCConvergence}
Assume that $f=g-h$ is bounded from below and that at least one of $g$ and $h$ is strongly convex.
Let $(u^k)$ be a sequence generated by the DC-Algorithm, \cref{GenDCA}, with the associated sequence $(s^k)$ of subgradients.

Suppose that there is a subsequence indexed by $(k_n)$ such that $u^{k_n} \to \bar u$ in $X$ and
$s^{k_n}\rightharpoonup^* \bar s$ in $X'$.
Then  $\bar s \in \partial g(\bar u) \cap \partial h(\bar u)$, i.e., $\bar u$ is a critical point of $f=g-h$.
\end{theorem}

\begin{proof}
This result is proven in \cite[Theorem 2.11]{Lim2023} for Hilbert spaces and can be easily generalized to Banach spaces using similar reasoning.
\end{proof}

In the following we show that for similar convergence results it suffices to solve the subproblem \eqref{generalDC-subproblem} only approximately.

\begin{algorithm}[Inexact DC-Algorithm]\label{DCA_inexact}
\begin{itemize}
\item[]
\item[\text{(S.0)}] Choose $u^0 \in X$ and set $k := 0$.
\item[\text{(S.1)}] Choose $s^k \in \partial h(u^k)$ and let $u^{k+1}$ be an $\varepsilon^k$-stationary point for the convex problem
\begin{align}
\min_{u \in X} g(u) - \langle s^k, u \rangle,
\end{align}
i.e., it holds $s^k + \varepsilon^k \in \partial g(u^{k+1})$.
\item[\text{(S.2)}] If a suitable termination criterion is satisfied: STOP.
\item[\text{(S.3)}] Set $k \leftarrow k+1$ and go to (S.1).
\end{itemize}
\end{algorithm}

Usually the inexactness in the DC-Algorithm is incorporated via the $\varepsilon$-subdifferential, see \cite{Oliveira2018,Thi2018}.
In our formulation $\varepsilon^k$ plays the role of a residuum to the necessary and sufficient condition $0 \in \partial g(u^{k+1}) - s^k$.
The definition of $\varepsilon^k$  implies
\[
\dist\left(0, \partial g(u^{k+1}) - s^k \right) \leq \Vert \varepsilon^k \Vert.
\]
Like in most inexact DC-Algorithms there is no guarantee that the subproblem can be solved in a finite number of iterations for a prescribed maximal value of $\Vert \varepsilon^k \Vert$.
For a detailed discussion see for example \cite{Zhang2023}.

We have the following inexact convergence result.

\begin{theorem}[Inexact DC]
Assume that $f=g-h$ is bounded from below and that at least one of $g$ and $h$ is strongly convex.
Let $(u^k)$ be a sequence generated by the inexact DC-Algorithm, \cref{DCA_inexact}, with the associated sequences $(s^k)$ and $(\varepsilon^k)$.

Suppose that there is a subsequence indexed by $(k_n)$ such that $u^{k_n} \to \bar u$ in $X$ and
$s^{k_n}\rightharpoonup^* \bar s$ in $X'$. Suppose in addition
\[
\sum_{k=1}^\infty \Vert \varepsilon^k \Vert_{X'}^2 < \infty.
\]
Then  $\bar s \in \partial g(\bar u) \cap \partial h(\bar u)$, i.e., $\bar u$ is a critical point of $f=g-h$.
\end{theorem}
\begin{proof}
From $s^k + \varepsilon^k \in \partial g(u^{k+1})$, and the uniform convexity of $g$ or $h$ we get
\begin{align*}
\frac{\alpha}{2} \Vert u^{k+1}-u^k \Vert_X^2 &\leq f(u^k) - f(u^{k+1}) + \langle \varepsilon^k, u^{k+1} - u^k \rangle \\
&\leq f(u^k) - f(u^{k+1}) + \frac1\alpha \Vert \varepsilon^k \Vert_{X'}^2 + \frac{\alpha}{4} \Vert u^{k+1}-u^k \Vert_X^2
\end{align*}
for some $\alpha>0$.
Since $g$ is proper and $\partial g(u^1) \ne \emptyset$ it follows $f(u^1) < \infty$.
As $f$ is bounded from below by some $M \in \R$, we get after summing the above inequality
\[
\frac{\alpha}{4} \sum_{k=1}^\infty\Vert u^{k+1}-u^k \Vert_X^2 \le f(u^1) - M + \frac1\alpha  \sum_{k=1}^\infty \Vert \varepsilon^k \Vert_{X'}^2 < +\infty.
\]
This implies $u^{k_n+1}\to \bar u$ in $X$.
As in the proof of \cite[Theorem 2.11]{Lim2023}, we can pass to the limit in the inclusions
$s^{k} + \varepsilon^{k} \in \partial g(u^{k+1})$ and $s^{k} \in \partial h(u^{k})$
along the subsequence $(k_n)$
to obtain $\bar s \in \partial g(\bar u) \cap \partial h(\bar u)$.
\end{proof}

\section{DC Reformulation in Function Space} \label{Section:PenalizedProblem}

In the following we now apply the results on the largest-$K$-norm to the general $L^0$-constrained problem, prove solvability and study the penalized problem.
For this purpose let $(\Om, \mathcal{A}, \mu)$ be again a $\sigma$-finite measure space.
We consider the general problem
\begin{align}\label{MinProblem}
\min_{u \in X} f(u) \quad \text{subject to} \quad \Vert u \Vert_0 \leq K
\end{align}
with $K \in (0, \mu(\Om))$ under the assumption that
\begin{itemize}
\item $X$ is a reflexive Banach space that is compactly embedded in $L^1(\Om)$,
\item $f: X \to \overline{\R}$ is proper, lower semicontinuous, convex, bounded from below and coercive, and
\item there is $u_0 \in X$ such that $f(u_0) < \infty$ and $\|u_0\|_0 \le K$.
\end{itemize}
By \cref{Theo:ReformSepa} problem \eqref{MinProblem} is equivalent to
\[
\min_{u \in X} f(u) \quad \text{subject to} \quad \Vert u \Vert_1 - \vert u \vert_K = 0.
\]
In the following, we want to study the associated penalized problem
\begin{align} \label{penMinProblem}
\min_{u \in X} f(u) + \rho(\Vert u \Vert_1 - \vert u \vert_K), \qquad \rho > 0,
\end{align}
where $\rho$ is a penalty parameter.
Note that this problem is a DC problem.

\subsection{A Penalty Method}

Let us define for every $u \in X$ the auxiliary functions
\begin{align*}
P(u) &:= \Vert u \Vert_1 - \vert u \vert_K, \\
f_{\rho}(u) &:= f(u) + \rho(\Vert u \Vert_1 - \vert u \vert_K) = f(u) + \rho P(u).
\end{align*}

\begin{lemma} \label{hcontinuous}
For every $K \in (0, \mu(\Om))$ the functions $\Vert \cdot \Vert_1$, $\vert \cdot \vert_K$, and $P$ are weakly continuous on $X$.
\end{lemma}
\begin{proof}
Since $X$ is compactly embedded in $L^1(\Om)$, this is a direct consequence of the continuity of all the functions with respect to the $L^1(\Om)$-norm by
\cref{Lem:Cont}.
\end{proof}

The weak continuity implies the global solvability of the original problem \eqref{MinProblem} and the penalized problem \eqref{penMinProblem}.

\begin{theorem} \label{TheoremMinGlobalSolution}
The $L^0$-constrained problem \eqref{MinProblem} has a global solution.
\end{theorem}

\begin{proof}
The global solvability follows by standard arguments:
Let $M := \{u \in X \mid \Vert u \Vert_0 \leq K \}$ be the feasible set of problem \eqref{MinProblem},
which is non-empty.
Let $(u_n) \subset M$ be a minimizing sequence, i.e., $\lim_{n \to \infty} f(u_n) = \inf_{u \in M} f(u)$.
Due to the assumptions on $f$, it follows $ \inf_{u \in M} f(u) \in \R$.
Since $f$ is bounded from below, it follows from the coercivity of $f$ that $(u_n)$ is bounded in $X$.
Since $X$ is reflexive, we can extract a weakly convergent subsequence $(u_{n_k})$ with $u_{n_k} \overset{X}{\rightharpoonup} u^*$ for some $u^* \in X$.
As $f$ is proper, convex and lower semicontinuous, and therefore weakly lower semicontinuous, we get
\[
f(u^*) \leq \liminf_{k \to \infty} f(u_{n_k}) = \lim_{k \to \infty} f(u_{n_k}) = \inf_{u \in M} f(u).
\]
By \cref{Theo:ReformSepa}
it holds $M = \{ u \in X \mid P(u) = 0\}$.
The weak continuity of $P$ by \cref{hcontinuous} implies $u^* \in M$, and $u^*$ solves \eqref{MinProblem}.
\end{proof}

\begin{lemma} \label{LemmaPenProblemSol}
For every $\rho > 0$ the penalized problem \eqref{penMinProblem} has a global solution.
\end{lemma}

\begin{proof}
The proof can be carried out using similar arguments as that of \cref{TheoremMinGlobalSolution}.
\end{proof}

For the corresponding penalty method, we have the following result.

\begin{theorem} \label{GenTheorem2}
Let $(\rho_n) > 0$ be monotonically increasing  with $\lim_{n \to \infty} \rho_n = \infty$. Let $u_n$ be a global minimum of the penalized problem \eqref{penMinProblem} with $\rho = \rho_n$ for all $n \in \N$.
Then the sequence $(u_n)$ satisfies
\begin{enumerate}
\item \label{pen_1} $f_{\rho_n}(u_n) \leq f_{\rho_{n+1}}(u_{n+1})$ for all $n \in \N$,
\item \label{pen_2} $P(u_n) \geq P(u_{n+1})$ for all $n \in \N$,
\item \label{pen_3} $f(u_n) \leq f(u_{n+1})$ for all $n \in \N$,
\item \label{pen_4} $\lim_{n \to \infty} P(u_n) = 0$,
\item \label{pen_5} $(u_n)$ is bounded in $X$,
\item \label{pen_6} every weak accumulation point $\bar u$ of the sequence $(u_n)$ is a global solution of \eqref{MinProblem}.
\end{enumerate}
\end{theorem}

\begin{proof}
The statements (\ref{pen_1})--(\ref{pen_3}) can be proven analogously to the proof for a finite dimensional penalty method, see, e.g., \cite[Satz 5.6]{Kanzow2002}.

(\ref{pen_4}) (\ref{pen_5})  By assumption, there is $u_0 \in X$ with $P(u_0)=0$ and $f(u_0)<\infty$.
Let $v \in X$ be a feasible point for \eqref{MinProblem}, which exists by assumption.
Since $u_n$ is a global minimum of $f_{\rho_n}$, we get
\[
f_{\rho_n}(u_n)   = f(u_n) + \rho_n P(u_n) \le f(v).
\]
Since $f$ is bounded from below and $P$ is non-negative, choosing $v = u_0$ implies $P(u_n) \to 0$ for $n\to\infty$.
The coercivity of $f$ implies boundedness of $(u_n)$ in $X$.

(\ref{pen_6})
Let $(u_{n_k})$ be a subsequence such that $u_{n_k} \overset{X}{\rightharpoonup} \bar u$.
Since $f$ is weakly lower semicontinuous, and $P$ is weakly continuous by \cref{hcontinuous}, it follows $f(\bar u) \leq \liminf_{k \to \infty} f(u_{n_k})$
and $P(\bar u) = \lim_{k \to \infty} P(u_{n_k})=0$.
Hence, $\bar u$ is feasible for \eqref{MinProblem}.
The inequality above implies $f(u_{n_k}) \le f(v)$.
Passing to the limit proves $f(\bar u) \le f(v)$
for all feasible $v \in X$, and $\bar u$ is a solution of  \eqref{MinProblem}.
\end{proof}

\subsection{Optimality Conditions}

In the convergence result for the penalty method \cref{GenTheorem2}, we assume that we have access to a sequence of global solutions of the penalized problem \eqref{penMinProblem}.
However, the DC-Algorithm in general does not deliver global solutions but only critical points, see \cref{Theo:DCConvergence}.
Therefore, we want to consider sequences of critical points of the penalized problem and
study their limiting behavior.

For this purpose, we first investigate optimality conditions of the penalized problem \eqref{penMinProblem}
under the additional assumptions that
\begin{itemize}
\item $X$ is dense in $L^1(\Omega)$, and
\item $f: X \to \R$ is Gateaux-differentiable.
\end{itemize}

The following auxiliary result shows that the subdifferential does not change if a dense subspace of $L^1(\Om)$ is considered.
This particularly applies to the largest-$K$-norm.

\begin{theorem}\label{theo:Subdiff0}
Let $f: L^1(\Om) \to \R$ be convex and continuous and $X \subseteq L^1(\Om)$ a dense subspace. Then it holds
\[
\partial^X f(0) = \partial^{L^1} f(0).
\]
\end{theorem}

\begin{proof}
The inclusion $\partial^{L^1} f(0) \subset \partial^X f(0)$ follows directly from the fact that $X \subset L^1(\Om)$.
Now let $s \in \partial^X f(0)$, i.e., $s$ is linear and continuous on $X$, and
\[
\langle s, v \rangle \leq f(v) \quad \forall v \in X.
\]
Because of the density of $X$ in $L^1(\Omega)$, we can uniquely extend $s$ to $L^1(\Om)$.
Moreover, there exists for every $u \in L^1(\Om)$ a sequence $(u_n) \subset X$ with $\Vert u_n - u \Vert_1 \to 0$.
It follows
\[
\langle s, u \rangle = \lim_{n \to \infty} \langle s, u_n \rangle \leq \lim_{n \to \infty} f(u_n) = f(u),
\]
which shows $s \in \partial^{L^1} f(0)$.
\end{proof}

Since $X \subseteq L^1(\Om)$ is dense, we know by \cref{theo:Subdiff0} that $\partial^X \vert \cdot \vert_K(u^*) =  \partial^{L^1} \vert \cdot \vert_K(u^*)$ and
$\partial^X \Vert \cdot \Vert_1(u^*) = \partial^{L^1} \Vert \cdot \Vert_1(u^*)$.
Therefore, we simply write $\partial \vert \cdot \vert_K(u^*)$ and $\partial \Vert \cdot \Vert_1(u^*)$ for these subdifferentials.

We consider the DC-decomposition of problem \eqref{penMinProblem} into $f + \rho \Vert \cdot \Vert_1$ and $\rho \vert \cdot \vert_K$.
Strongly critical points give rise to the following necessary optimality condition.

\begin{corollary}\label{Theo:DCnecessary}
Let $u^*$ be a local minimum or a strongly critical point of problem \eqref{penMinProblem}.
Then for every $s \in \partial \vert \cdot \vert_K(u^*)$ there exists $r \in \partial \Vert \cdot \Vert_1(u^*)$ such that
\[
\rho \cdot \langle s,v\rangle = \langle f'(u^*), v \rangle + \rho \cdot \langle r,v\rangle
\]
for every $v \in X$.
\end{corollary}

\begin{proof}
This is a direct consequence of \cref{prop:DCnecessary}.
\end{proof}

Let us now utilize the developments of the subgradient of $|\cdot|_K$ from \cref{theo:Subdiff} and \cref{Theo:GeneralSubdiff}.

\begin{theorem}\label{cor_optimality_cond}
Let $u^*$ be a local minimum or strongly critical point of problem \eqref{penMinProblem}.
Let $A \subset \supp u^*$ be given such that $\mu(A) \le K$ and $\int_A \vert u^* \vert \dm = \vert u^* \vert_K$.
Then the following statements are satisfied:
\begin{enumerate}
\item \label{optcon_pen_1} $f'(u^*) \in L^\infty(\Om)$,
\item \label{optcon_pen_2} $f'(u^*)(x) = 0$ for almost every $x \in \supp u^* \cap A$,
\item \label{optcon_pen_3} $f'(u^*)(x) = - \rho \sign(u^*)(x)$ for almost every $x \in \supp u^* \setminus A$,
\item \label{optcon_pen_4} $\vert f'(u^*)(x) \vert \leq \rho$ for almost every $x \not\in \supp u^*$.
\end{enumerate}
In addition, $A$ is unique up to sets of measure zero, i.e., if $B \subset \supp u^*$ satisfies $\mu(B) \le K$ and $\int_B \vert u^* \vert \dm = \vert u^* \vert_K$
then $\mu(A \triangle B)=0$, where $A \triangle B$ denotes the symmetric difference of the sets $A,B$.
\end{theorem}

\begin{proof}
Define $s:= \sign(u^*) \chi_{A}$. Then it follows $\Vert s \Vert_0 \leq K$ and  $s \in \partial \vert \cdot \vert_K(u^*)$ by \cref{Theo:GeneralSubdiff}.
By \cref{Theo:DCnecessary}, there exists $r \in \partial \Vert \cdot \Vert_1(u^*)$ such that
\begin{align}\label{equ:optimality_cond}
\rho \int_{A} \sign(u^*) \cdot v \dm = \langle f'(u^*), v \rangle + \rho \int_\Om r \cdot v \dm
\end{align}
for every $v \in X$.
Analogously to the proof of \cref{theo:Subdiff0}, we can argue that $f'(u^*) \in L^\infty(\Om)$.
The density of $X$ in $L^1(\Om)$ and the continuity of \eqref{equ:optimality_cond} with respect to $v$ show that \eqref{equ:optimality_cond} holds for every $v \in L^1(\Om)$.

Testing with $v = \chi_B \in L^1(\Om)$ for $B \subseteq A $ yields
\[
\rho \int_B \sign(u^*) \dm = \int_B f'(u^*) \dm + \rho \int_B \sign(u^*) \dm,
\]
since $r = \sign(u^*)$ a.e.\@ on $\supp u^*$. This shows $f'(u^*) = 0$ a.e.\@ on $A$.

On the other hand, testing with $v = \chi_B u^*$ for $B \subseteq \supp u^* \setminus A$ gives
\[
0 = \int_B f'(u^*) \cdot u^* \dm + \rho \int_B \vert u^* \vert \dm,
\]
which shows that $f'(u^*)(x) = - \rho \sign(u^*)(x)$ for almost every $x \in \supp u^* \setminus A$.

Lastly, testing with $v = \chi_B$ for $B \subseteq (\supp u^*)^c$ with $\mu(B) < \infty$ results in
\[
0 = \int_B f'(u^*) \dm + \rho \int_B r \dm.
\]
Since $\Vert r \Vert_\infty \leq 1$, this shows $\vert f'(u^*)(x) \vert \leq \rho$ for almost every $x \not\in \supp u^*$.

Let now $B \subset \supp u^*$ be such that $\mu(B) \le K$ and $\int_B \vert u^* \vert \dm = \vert u^* \vert_K$.
Then (\ref{optcon_pen_2}) and (\ref{optcon_pen_3}) imply $0 = f'(u^*) = - \rho \sign(u^*)$ almost everywhere on $A \triangle B$.
This implies $\mu(A\triangle B)=0$.
\end{proof}

For an atom-free measure space this can be strengthened by examining the upper level sets.

\begin{corollary}\label{cor_optimality_cond_atomfree}
Let $(\Om, \mathcal{A}, \mu)$ be atom-free.
Let $u^*$ be a local minimum or strongly critical point of problem \eqref{penMinProblem}.
Then it either holds $\Vert u^* \Vert_0 \leq K$ or there is $t > 0$ such that $\mu(\Om_{>t}) = K$ or $\mu(\Om_{\geq t}) = K$.
Moreover, it holds
\begin{enumerate}
\item \label{optcon_pen_af_2} $f'(u^*)(x) = 0$ for almost every $x \in \{ |u^*| > t\}$ (or $x \in \{ |u^*| \ge t\} \cap \supp u^*$),
\item \label{optcon_pen_af_3} $f'(u^*)(x) = - \rho \sign(u^*)(x)$ for almost every $x \in \{0< |u^*| \leq t\}$ (or $x \in \{0< |u^*| < t\}$),
\item \label{optcon_pen_af_4} $\vert f'(u^*)(x) \vert \leq \rho$ for almost every $x \not\in \supp u^*$,
\end{enumerate}
where $t = 0$ if $\Vert u^* \Vert_0 \leq K$.
\end{corollary}
\begin{proof}
By \cref{Goal1}, there exist $t \geq 0$ and $\Om_{>t} \subset A \subset \Om_{\geq t}$ with $\mu(A) = K$ and $\int_A \vert u^* \vert \dm = \vert u^* \vert_K$.
If $t = 0$ we have that $\Vert u^* \Vert_0 \leq K$ and get from \cref{cor_optimality_cond} that
\begin{enumerate}
\item $f'(u^*)(x) = 0$ for almost every $x \in \supp u^*$,
\item $\vert f'(u^*)(x) \vert \leq \rho$ for almost every $x \not\in \supp u^*$.
\end{enumerate}
Let $t > 0$ and suppose that $\mu(\Om_{>t}) < K < \mu(\Om_{\geq t})$.
Then there exist sets $A,B$ with $\Om_{>t} \subset A,B \subset \Om_{\geq t}$, $\mu(A) = K = \mu(B)$ and $\int_A \vert u^* \vert \dm = \vert u^* \vert_K = \int_B \vert u^* \vert \dm$ such that $\mu(A \triangle B) > 0$, a contradiction to \cref{cor_optimality_cond}.
Hence, $\mu(\Om_{>t}) = K$ or $\mu(\Om_{\geq t}) = K$.
Applying \cref{cor_optimality_cond} to $\Om_{>t}$ or $\Om_{\geq t}$ respectively shows the claim.
\end{proof}

\cref{cor_optimality_cond} implies the following exact penalization result.

\begin{corollary}\label{cor_exact_penalty}
If $u^*$ is a strongly critical point of problem \eqref{penMinProblem} with $\rho > \Vert f'(u^*)\Vert_\infty$,
then $u^*$ is feasible for the original problem, i.e., it holds $\Vert u^* \Vert_0 \leq K$.
Moreover, $u^*$ satisfies
\[
\langle f'(u^*), u^* \rangle=0.
\]
If in addition $u^*$ is a global minimum of problem \eqref{penMinProblem} with $\rho > \Vert f'(u^*)\Vert_\infty$, then $u^*$ solves the original $L^0$-constrained problem \eqref{MinProblem}.
\end{corollary}
\begin{proof}
Let $A \subset \supp u^*$ be given such that $\mu(A) \le K$ and $\int_A \vert u^* \vert \dm = \vert u^* \vert_K$.
Then condition (\ref{optcon_pen_3}) of  \cref{cor_optimality_cond}
implies $\mu ( \supp u^* \setminus A) =0$. Hence, $\Vert u^* \Vert_0 \leq K$.
Using $v = u^*$ in \cref{Theo:DCnecessary} proves $\langle f'(u^*), u^* \rangle=0$ since $\langle  s, u^* \rangle = \vert u^* \vert_K = \Vert u^* \Vert_1 = \langle r, u^*\rangle$ by \cref{lem:Subdiff} and \cref{Theo:ReformSepa}.
The last claim follows directly from $\Vert u^* \Vert_1 - \vert u^* \vert_K = 0$.
\end{proof}

\begin{remark}\label{Remark:Exact_Pen}
Note that \cref{cor_exact_penalty} implies the exact penalization result of \cite[Theorem 4, Corollary 4]{Gotoh2017}.
Let $u^*$ be a global solution of
\[
\min_{u \in \R^n} f(u) + \rho \left( \Vert u \Vert_1 - \vert u \vert_K \right),
\]
with $\rho > 0$, $f(u) := \frac{1}{2} u^T Q u + q^T u$ for $Q \in \R^{n \times n}$ symmetric and positive definite and $q \in \R^n$.
Then  $u^*$
satisfies $\Vert u^* \Vert_2 \leq \frac{2 \Vert q \Vert_2}{\lambda_{\text{min}}(Q)}$, \cite[Appendix A4]{Gotoh2017}.
By \cref{cor_exact_penalty}, we have the exact penalization if $\rho> \rho^*$, where $\rho^*$ is defined as
\[
\Vert f'(u^*) \Vert_\infty \leq \Vert Q u^* + q\Vert_2 \leq \Vert Q \Vert_2 \Vert u^* \Vert_2 + \Vert q \Vert_2
\leq \Vert q \Vert_2 \left(\frac{2 \Vert Q \Vert_2}{\lambda_{\text{min}}(Q)} + 1 \right) =: \rho^*,
\]
which is similar to \cite[Theorem 4, Corollary 4]{Gotoh2017}.
\end{remark}

The conclusions of \cref{cor_optimality_cond} only hold for strongly critical points.
However, the DC-Algorithm generates iterates that in general converge towards critical points only, cf. \cref{Theo:DCConvergence}.
We can derive the following necessary condition for critical points.

\begin{lemma}\label{Lem:CriticalNecessary}
Let $u^*$ be a critical point of problem \eqref{penMinProblem} with $s \in \partial \vert \cdot \vert_K(u^*)$ and $r \in \partial \Vert \cdot \Vert_1(u^*)$. Then it holds
\[
f(u^*) + \rho \left(\Vert u^* \Vert_1 - \vert u^* \vert_K\right) \leq f(v) + \rho \langle r-s, v \rangle
\]
for all $v \in X$.
\end{lemma}

\begin{proof}
By criticality, it holds $f'(u^*) = \rho(s-r)$.
Moreover, it holds $\langle  s, u^* \rangle = \vert u^* \vert_K$ and $\langle r, u^*\rangle = \Vert u^* \Vert_1$ by \cref{lem:Subdiff}.
Since $f$ is convex, it follows for $v\in X$
\[
f(v) - f(u^*) \geq \rho \langle s-r, v-u^* \rangle \geq \rho\langle s-r, v \rangle + \rho \left(\Vert u^* \Vert_1 - \vert u^* \vert_K \right),
\]
which is the claim.
\end{proof}

\begin{corollary}
Let $(\Om, \mathcal{A}, \mu)$ be atom-free and
$u$ be a critical point of problem \eqref{penMinProblem} with $s \in \partial \vert \cdot \vert_K(u)$ and $r \in \partial \Vert \cdot \Vert_1(u)$.
Take $t\geq 0$ from \cref{Goal1} and define
\[
 u_t := \begin{cases} u-t & \text{ if } u > t, \\
         0 &   \text{ if } |u| \le t, \\
         u+t & \text{ if } u <- t, \\
        \end{cases}
\]
then $\Vert u_t \Vert_0 \leq K$, $\langle r - s, u_t \rangle = 0$ and $f(u) + \rho(\Vert u \Vert_1 - \vert u \vert_K) \leq f(u_t)$.
\end{corollary}
\begin{proof}
First, note that by definition $\supp u_t \subset \Om_{>t}$ and hence $\Vert u_t \Vert_0 \leq K$.
By the proof of \cref{TheoremAltKRepr}, we have $s = \sign(u) = r$ a.e.\@ on $\Om_{>t}$.
Therefore, it holds $\int_\Om r u_t \dm = \int_\Om s u_t \dm$.
Using the inequality of \cref{Lem:CriticalNecessary} with $v = u_t$ shows the claim.
\end{proof}

If a critical point of the penalized problem is feasible, we have the following result.
\begin{corollary}\label{Cor:Critical}
Let $u^*$ be a critical point of problem \eqref{penMinProblem} that is feasible for the original problem \eqref{MinProblem}, i.e., $\Vert u^* \Vert_0 \leq K$.
Then it holds
\[
\langle f'(u^*), u^* \rangle=0.
\]
Moreover, we have $f(u^*) \leq f(v)$ for all $v \in X$ with $\supp v \subset \supp u^*$.
\end{corollary}
\begin{proof}
Since $u^*$ is a critical point, there exist $s \in \partial \vert \cdot \vert_K (u^*)$ and $r \in \partial \Vert \cdot \Vert_1 (u^*)$ such that $0 = f'(u^*) + \rho(r-s)$.
Testing this equation with $u^*$ proves the claim, where we have used
$ \langle r-s , u^*\rangle = \|u^*\|_1 - |u^*|_K =0$, see \cref{lem:Subdiff} and \cref{Theo:ReformSepa}.
Since $\Vert u^* \Vert_0 \leq K$, it holds $r = \sign(u^*) = s$ a.e.\@ on $\supp u^*$.
Hence, for $v \in X$ with $\supp v \subset \supp u^*$ it holds $\langle r - s, v \rangle = 0$ and  \cref{Lem:CriticalNecessary} shows the claim.
\end{proof}

In the general case,
we can prove that weak limit points are feasible. In addition, we have a weaker variant of the condition $ \langle f'(u^*), u^* \rangle=0$.

\begin{theorem}
Let $(\rho_n)> 0$ be given with $\lim_{n \to \infty} \rho_n = \infty$, and let
$(u_n) \subseteq X$ be a sequence of critical points of problem \eqref{penMinProblem} with $\rho = \rho_n$.

Then the sequence $(u_n)$ is bounded in $X$ and it holds $\Vert u_n \Vert_1 - \vert u_n \vert_K \to 0$.
Moreover, every weak accumulation point $\bar u$ is feasible for the original problem \eqref{MinProblem}, i.e., $\Vert \bar u \Vert_0 \leq K$,  and satisfies $f(\bar u) \leq f(0)$.
In addition, it holds
\[
\langle f'(u_n), u_n \rangle \in [\inf_{v \in X} f(v) - f(0),\, 0] \quad \forall n \in \N.
\]
\end{theorem}
\begin{proof}
Using $v=0$ in  \cref{Lem:CriticalNecessary} we get
\begin{equation}\label{Inequ:PenaltyCritical}
\inf_{v \in X} f(v) + \rho_n \left(\Vert u_n \Vert_1 - \vert u_n \vert_K\right) \leq f(u_n) + \rho_n \left(\Vert u_n \Vert_1 - \vert u_n \vert_K\right) \leq f(0)
\end{equation}
for every $n \in \N$.
Suppose that $(u_n)$ is not bounded.
Then it holds $\Vert u_n \Vert_X \to \infty$ at least on a subsequence.
Since $f$ is coercive and bounded from below, we get the boundedness of $(u_n)$ in $X$.
Since $X$ is reflexive, this implies the existence of a weak accumulation point.
Moreover, $\rho_n \to \infty$  implies $\Vert u_n \Vert_1 - \vert u_n \vert_K \to 0$.

Let $\bar u$ be a weak accumulation point of $(u_n)$.
W.l.o.g. we can assume that the whole sequence converges weakly to $\bar u$.
Then the weak continuity of $\vert \cdot \vert_K$ by \cref{hcontinuous} implies $\Vert \bar u \Vert_1 - \vert \bar u \vert_K = 0$, i.e., $\Vert \bar u \Vert_0 \leq K$.

The inequality \eqref{Inequ:PenaltyCritical} directly shows $\rho_n \left(\Vert u_n \Vert_1 - \vert u_n \vert_K\right) \in [0, f(0) - \inf_{v \in X} f(v)]$.
Let $s_n \in \partial \vert \cdot \vert_K(u_n)$ and $r_n \in \partial \Vert \cdot \Vert_1(u_n)$ be the subgradients for $u_n$ that satisfy criticality, i.e., $0 = f'(u_n) + \rho_n(r_n - s_n)$.
Testing this equality with $u_n$ provides
\[
0 = \langle f'(u_n), u_n \rangle + \rho_n \left( \Vert u_n \Vert_1 - \vert u_n \vert_K \right),
\]
which shows the claim.
Since $f$ is weakly lower semicontinuous, passing to the limit in \eqref{Inequ:PenaltyCritical}  shows
\[
f(\bar u) \leq \liminf_{n \to \infty} f(u_n) + \rho_n \left(\Vert u_n \Vert_1 - \vert u_n \vert_K\right) \leq f(0).
\]
\end{proof}

\section{Aspects of Discretization}\label{Section:ExampleProblem}

Let $\Om \subseteq \R^d$ be a bounded domain. In this section we will work with the Lebesgue measure, which is atom-free, see, e.g., \cite[211M]{Fremlin2003}.
We are interested in numerically solving the $L^0$-constrained problem
\begin{align} \label{Problem}
\min_{u \in H_0^1(\Om)} \frac{1}{2} \Vert \nabla u \Vert_{L^2(\Om)}^2 - \int_{\Om} g u \dm \quad \text{subject to} \quad \Vert u \Vert_0 \leq K
\end{align}
for $g \in L^2(\Om)$ and $K \in (0, \mu(\Om))$.
This is motivated by a shape optimization problem considered in \cite{ButtazzoMaialeVelichkov2021}.
In addition, this problem serves as a prototype for more complicated problems including optimal control problems.

Note that the assumptions of \cref{Section:PenalizedProblem} are satisfied for problem \eqref{Problem} and hence \eqref{Problem} is globally solvable by \cref{TheoremMinGlobalSolution}.
Moreover, we can prove that the assumptions of \cref{Theo:DCConvergence} are satisfied.
We even have strong convergence of the subgradients.

\begin{theorem}
Let $(u^k)$ be a sequence generated by the DC-Algorithm, \cref{GenDCA}, applied to the penalization of problem \eqref{Problem} with the associated sequence $(s^k)$ of subgradients.
Then the assumptions of \cref{Theo:DCConvergence} are met, i.e., the objective is a DC function which is bounded from below such that at least one convex component is strongly convex and there is a subsequence indexed by $(k_n)$ such that $u^{k_n} \to \bar u$ in $H_0^1(\Om)$ and
$s^{k_n} \to \bar s$ in $H^{-1}(\Om)$.
\end{theorem}
\begin{proof}
The first term of the objective functional is the squared norm of the Hilbert space $H_0^1(\Om)$ and hence strongly convex.
The boundedness from below follows directly from Poincare's inequality, see, e.g., \cite[Satz 6.13]{Dobrowolski2010}.

Let $r^k \in \partial \Vert \cdot \Vert_1(u^k)$ be such that
\[
-\Delta u^{k+1} = g - \rho(r^k - s^k),
\]
which exists since $u^{k+1}$ minimizes the DC subproblem.
Here, $\rho > 0$ is the penalty parameter.
By \cref{lem:Subdiff}, the sequences $(s^k)$ and $(r^k)$ are bounded in $L^\infty(\Om)$.
Since $\Om$ is bounded, $L^\infty(\Om)$ is continuously embedded into $L^2(\Om)$ and $L^2(\Om)$ is compactly embedded into $H^{-1}(\Om)$.
This shows that there exists a strongly convergent subsequence of the right-hand side of the above equation in $H^{-1}(\Om)$.

The strong convergence of the associated subsequence of $(u^{k+1})$ in $H_0^1(\Om)$ now follows by the continuity of the map $(-\Delta)^{-1}: H^{-1}(\Om) \to H_0^1(\Om)$, cf. \cite[Satz 6.29]{Dobrowolski2010}.
Hence, there is a subsequence indexed by $(k_n)$ such that $u^{k_n} \to \bar u$ in $H_0^1(\Om)$.
Because of the boundedness of the associated subsequence of subgradients $(s^{k_n})$ in $L^\infty(\Om)$, the claim follows analogously to above after possibly extracting another strong convergent subsequence in $H^{-1}(\Om)$.
\end{proof}

At first we take a closer look at the discretization of the $L^0$-constraint in the context of finite element methods.

\subsection{Reformulation of $L^0$-Constraints in the Space of Piecewise Linear Continuous Functions} \label{section:DiscPiecewiseLin}

We consider a standard finite element discretization with simplices.
Let $\mathcal{T}$ be a shape-regular mesh with $d$-simplex elements $T_i$, $i=1,...,m$,
and nodal basis functions $\varphi_j$, $j = 1,...,N$, which generate the space $V_N$ of piecewise linear continuous functions $u_h$.

Define the vector $\mu(\mathcal{T}) \in \R^m$ by
\[
\mu(\mathcal{T})_i := \mu(T_i),
\]
and the matrix $D \in \R^{m \times N}$ by
\[
D_{i,j} := \begin{cases}
1 & \text{ if the node } j \text{ is a vertex of the element } T_i, \\
0 & \text{ otherwise}.
\end{cases}
\]
Given $u_h \in V_N$, we denote by $u \in \R^N$ the vector of coefficients with respect to the basis $(\varphi_j)$.
Furthermore, for $u_h \in V_N$ we define the vector
\[
w_u := D \vert u \vert \in \R^m,
\]
where the absolute value is taken elementwise.
Therefore, $w_{u,i}$ is the sum of the absolute values of $u_h$ at the $d+1$ vertices of the element $T_i$.

For these vectors, we consider (pseudo-)norms weighted by $\mu(\mathcal{T})$, i.e.,
\[
 \Vert w_u \Vert_{0, \mu(\mathcal{T})} := \sum_{i=1}^m \mu(\mathcal{T})_i |w_{u,i}|_0,
\]
and similarly for $\|w_u\|_{1,\mu(\mathcal{T})}$ and $|w_u|_{K,\mu(\mathcal{T})}$,
which corresponds to defining the measure $\mu$ on $\{1, \dots, m\}$ by $\mu(\{i\}) = \mu(\mathcal T)_i$.
Using these weighted (pseudo-)norms,
we can express the $L^0(\Om)$-pseudo-norm of a function $u_h \in V_N$.

\begin{lemma}\label{Lem:L0refDis}
Let $u_h \in V_N$, then $\Vert u_h \Vert_0 = \Vert w_u \Vert_{0, \mu(\mathcal{T})}$.
\end{lemma}

\begin{proof}
Let $T_i \in \mathcal T$ be a simplex.
Since $u_h$ is piecewise linear and continuous, it holds either $u_h = 0$ a.e.\@ on $T_i$ or $u_h \neq 0$ a.e.\@ on $T_i$. Moreover, $u_h = 0$ a.e.\@ on $T_i$ is the case if and only if $u_h = 0$ at all vertices of $T_i$, i.e., if $w_{u,i} = 0$. This shows the claim.
\end{proof}

Then the following  equivalent reformulation holds for the discretized $L^0$-constraint $\Vert u_h \Vert_0 \leq K$.

\begin{theorem}\label{Theo:DiscreteReform}
Let $u_h \in V_N$ and $K \in (0, \mu(\Om))$. Then it holds
\[
\Vert u_h \Vert_0 \leq K \quad \Leftrightarrow \quad \Vert w_u \Vert_{1, \mu(\mathcal{T})} - \vert w_u \vert_{K, \mu(\mathcal{T})} = 0.
\]
\end{theorem}

\begin{proof}
Due to \cref{Lem:L0refDis}, it holds $\Vert u_h \Vert_0 \leq K$ if and only if $\Vert w_u \Vert_{0, \mu(\mathcal{T})} \le K$.
By \cref{Theo:ReformSepa}, this is equivalent to $\Vert w_u \Vert_{1, \mu(\mathcal{T})} - \vert w_u \vert_{K, \mu(\mathcal{T})} = 0$.
\end{proof}

Note that this result is based on the general $L^0$-reformulation \cref{Theo:ReformSepa} since the underlying measure space is purely atomic.

Let us define the following mesh-dependent approximations of the $L^1(\Om)$- and the largest-$K$-norm.

\begin{definition}\label{Def:Approx}
For every $u_h \in V_N$ define
\begin{align*}
\Vert u_h \Vert_{1,h} &:= \sum_{i=1}^N \vert u_i \vert \int_{\Om} \varphi_i \d\mu, \\
\vert u_h \vert_{K, h} &:= \max_{I \subseteq \{1,...,m\}: \sum_{i \in I} \mu(T_i) \leq K} \sum_{j=1}^N \vert u_j \vert \int_{\cup_{i \in I} T_i} \varphi_j \dm.
\end{align*}
\end{definition}

$\Vert u_h \Vert_{1,h}$ is a well-known  approximation for the $L^1(\Om)$-norm, see, e.g.,  \cite{Wachsmuth2010}.
These approximations connect to our reformulation in the following way.

\begin{theorem}\label{Theo:Approx}
For every $u_h \in V_N$ it holds
\begin{align*}
\Vert u_h \Vert_{1,h} &= \frac{1}{d+1} \Vert w_u \Vert_{1,\mu(\mathcal{T})}, \\
\vert u_h \vert_{K, h} &= \frac{1}{d+1} \vert w_u \vert_{K,\mu(\mathcal{T})}.
\end{align*}
\end{theorem}

\begin{proof}
If $j$ is a vertex of $T_i$, the integral $\int_{T_i} \varphi_j \dm$ equals the $d+1$-dimensional volume of a hyper-pyramid with base $T_i$ and height $1$.
Therefore, we get
\[\begin{split}
\int_{T_i} \varphi_j \dm
&= \begin{cases}
\frac{1}{d+1} \mu(T_i) & \text{ if } j \text{ is a vertex of } T_i, \\
0 & \text{ otherwise},
\end{cases} \\
&= \frac{1}{d+1} D_{i,j} \mu(T_i).
\end{split}\]
Let $I \subset \{1,...,m\}$. Then we get
\[
 \sum_{j=1}^N \vert u_j \vert \sum_{i \in I} D_{i,j} \mu(T_i)
 = \sum_{i \in I} \mu(T_i) \sum_{j=1}^N D_{i,j} \vert u_j \vert
 =\sum_{i \in I} \mu(T_i) w_{u,i} .
 \]
This yields
\begin{align*}
\vert u_h \vert_{K, h} &= \frac{1}{d+1} \max_{I \subseteq \{1,...,m\}: \sum_{i \in I} \mu(T_i) \leq K} \sum_{j=1}^N \vert u_j \vert \sum_{i \in I} D_{i,j} \mu(T_i) \\
&= \frac{1}{d+1} \max_{I \subseteq \{1,...,m\}: \sum_{i \in I} \mu(T_i) \leq K} \sum_{i \in I} \mu(T_i) w_{u,i} = \frac{1}{d+1} \vert w_u \vert_{K,\mu(\mathcal{T})}.
\end{align*}
Setting $K = \mu(\Om)$ proves the claim for $\Vert u_h \Vert_{1,h}$.
\end{proof}

Applying \cref{Theo:DiscreteReform} to the finite element discretization
\begin{align} \label{DisProblem}
\min_{u_h \in V_N} \frac{1}{2} \Vert \nabla u_h \Vert_{L^2(\Om)}^2 - \int_{\Om} g u_h \dm \quad \text{subject to} \quad \Vert u_h \Vert_0 \leq K
\end{align}
of problem \eqref{Problem} results in the finite dimensional problem
\[
\min_{u \in \R^N} \frac{1}{2} u^T A u - b^T u \quad \text{subject to} \quad \Vert w_u \Vert_{1,\mu(\mathcal{T})} - \vert w_u \vert_{K,\mu(\mathcal{T})} = 0.
\]
Here, $A$ and $b$ denote the stiffness matrix and load vector of the finite element discretization.
We are interested in solving the associated penalized problem
\begin{align} \label{PenProblem}
\min_{u \in \R^N} \frac{1}{2} u^T A u - b^T u + \rho \left(\Vert w_u \Vert_{1,\mu(\mathcal{T})} - \vert w_u \vert_{K,\mu(\mathcal{T})}\right)
\end{align}
for large $\rho > 0$.
Note that this problem is always globally solvable.
The exact penalty results \cite[Theorem 4, Corollary 4]{Gotoh2017}, see also \cref{cor_exact_penalty} and \cref{Remark:Exact_Pen}, can be adapted to the weighted norms in \eqref{PenProblem}:
There is $\rho^* > 0$ such that for all $\rho > \rho^*$ any solution of the penalized problem \eqref{PenProblem} is a solution of the original discretized problem \eqref{DisProblem}.
Unfortunately, $\rho^*$ depends on the discretization, and $\rho^*\to \infty$ for $h\searrow0$, where $h$ is the mesh-size.

\subsection{Convex Subdifferentials in Finite Dimensions}

Solving the penalized problem \eqref{PenProblem} by a DC-Algorithm requires the computation of a subgradient of the function $u \mapsto \vert w_u \vert_{K, \mu(\mathcal{T})}$.
The following results give  explicit characterizations of the necessary subdifferentials, compare also \cref{Theo:GeneralSubdiff}.

\begin{lemma}\label{lemmaSubdiff}
Let $K \geq 0$, $\lambda \in \R_{>0}^n$ and $x \in \R^n$.
Moreover, let $I^*$ be an optimal set for the maximization problem in $\vert x \vert_{K,\lambda}$ and $I_0^* := \{ i \in I^* \mid x_i = 0 \}$.
Then it holds
\[
\partial \vert \cdot \vert_{K, \lambda}(x) \supseteq \left\lbrace s \in \R^n \mid s_i
\begin{cases}
= \lambda_i\sign(x_i) & \text{ if } i \in I^* \setminus I_0^*, \\
\in [-\lambda_i , \lambda_i] & \text{ if } i \in I_0^*, \\
= 0 & \text{ if } i \notin I^*.
\end{cases}
\right\rbrace.
\]
\end{lemma}

\begin{proof}
Let $s$ be in the set on the right-hand side.
Then for any vector $y \in \mathbb{R}^n$ it holds
\[
s^T y = \sum_{i=1}^n s_i y_i \leq \sum_{i \in I^*} \vert s_i \vert \vert y_i \vert \leq \sum_{i \in I^*} \lambda_i \vert y_i \vert \leq \vert y \vert_{K, \lambda}, \quad s^T x = \sum_{i \in I^*} \lambda_i \vert x_i \vert = \vert x \vert_{K, \lambda},
\]
which shows the claim.
\end{proof}
For a complete characterization of the subdifferential of the finite-dimensional largest-$K$-norm see \cite[Theorem 3]{Watson1992}.

\begin{theorem}\label{theoSubdiff}
Let $K \geq 0$, $\lambda \in \R_{>0}^m$ and $u \in \R^N$. Then
\[
\partial \vert w_u \vert_{K, \lambda}(u) = \{ (\rho_i s_i)_i \in \R^N \mid \rho = D^T r,\: r \in \partial \vert \cdot \vert_{K, \lambda}(w_u), \: s_i \in \partial\vert \cdot \vert(u_i) \:\forall i \}.
\]
\end{theorem}

\begin{proof}
Define
\begin{align*}
g_K: \R^N \to \R, \quad &x \mapsto \vert Dx \vert_{K, \lambda},\\
f_i: \R^N \to \R, \quad &x \mapsto \vert x_i \vert \qquad \text{for every } i \in \{1,...,N\}, \\
F: \R^N \to \R^N, \quad &x \mapsto (f_1(x), ..., f_N(x))^T = \vert x \vert.
\end{align*}
Then $g_K$ is convex and componentwise non-decreasing on $\R_{\geq 0}^N$ and $f_i$ is convex for every $i \in \{1,...,N\}$.
Moreover, it holds $(g_K \circ F)(u) = \vert w_u \vert_{K, \lambda}$ for every $u \in \R^N$.
Therefore, it follows by the chain rule for the convex subdifferential \cite[see][Theorem 4.3.1]{Hiriart2001} that
\[
\partial \vert w_u \vert_{K, \lambda}(u) = \left\lbrace \sum_{i=1}^N \rho_i s^i \mid (\rho_1,...,\rho_N)^T \in \partial g_K(F(u)),\: s^i \in \partial f_i(u) \:\forall i \in \{1,...,N\} \right\rbrace.
\]
Furthermore, we have
\[
\partial g_K(x) = D^T \partial \vert \cdot \vert_{K, \lambda}(Dx)
\]
for every $x \in \R^N$ by \cite[Theorem 4.2.1]{Hiriart2001}. Obviously, it holds
\[
\partial f_i(u) = \left\lbrace s \in \R^N \mid s_i \in \partial\vert \cdot \vert(u_i), \: s_j = 0 \text{ for all } j \neq i \right\rbrace \text{ for every } i \in \{1,...,N\},
\]
which finishes the proof.
\end{proof}

Using \cref{lemmaSubdiff} to determine $r \in \partial \vert \cdot \vert_{K, \lambda}(w_u)$, this result provides a subgradient $s \in \partial \vert w_u \vert_{K, \lambda}(u)$ for the DC-Algorithm.
However, to be able to use \cref{lemmaSubdiff} numerically, one has to determine a set $I^* \subseteq \{1,...,n\}$, which attains the maximum in
\[
\max_{I \subseteq \{1,...,n\}: \sum_{i \in I} \lambda_i \leq K} \sum_{i \in I} \lambda_i \vert x_i \vert.
\]
Such problems are known as knapsack problems, see, e.g., \cite{Martello1990}.
In our implementation we used a greedy algorithm, cf. \cite[Section 2.4]{Martello1990}.

To solve the subproblem \eqref{DC-subproblem} of the DC-Algorithm we also need the convex sub\-differential of the function $u \mapsto \Vert w_u \Vert_{1, \mu(\mathcal{T})}$.

\begin{theorem}\label{CorSubdiff}
Let $u \in \R^N$. Then
\[
\partial \Vert w_u \Vert_{1, \mu(\mathcal{T})}(u) = \left\lbrace s \in \R^N \mid s_i \in \mu(\triangle_i) \partial \vert \cdot \vert(u_i)
\ \forall i \in \{1,...,N\} \right\rbrace,
\]
where $\triangle_i$ is the union of all simplices in the mesh $\mathcal{T}$ of which the node $i$ is a vertex, i.e., $\mu(\triangle_i) = \sum_{j=1}^m  D_{j,i} \mu(T_j)$.
\end{theorem}

\begin{proof}
It holds
\[
\Vert w_u \Vert_{1, \mu(\mathcal{T})} = \sum_{i=1}^m \mu(T_i) \sum_{j=1}^N D_{i,j} \vert u_j \vert = \sum_{j=1}^N \vert u_j \vert \sum_{i=1}^m  D_{i,j} \mu(T_i) = \sum_{j=1}^N \vert u_j \vert \mu(\triangle_j).
\]
Hence, $\Vert w_u \Vert_{1, \mu(\mathcal{T})}$ equals a weighted $\ell^1$-norm of $u$, which proves the claim.
\end{proof}

\subsection{Algorithms}

The explicit DC-Algorithm for problem \eqref{PenProblem} is:

\begin{algorithm} [DC-Algorithm for problem \eqref{PenProblem}] \label{DC-Algorithm}
\begin{itemize}
\item[]
\item[\text{(S.0)}] Choose $u^0 \in \R^N$ and set $k := 0$.
\item[\text{(S.1)}] Define $s^k := \rho s$ with $s \in \partial \vert w_u \vert_{K, \mu(\mathcal{T})}(u^k)$ and determine $u^{k+1}$ as a solution to
\begin{align} \label{DC-subproblem}
\min_{u \in \R^N} \frac{1}{2} u^T A u - b^T u + \rho \Vert w_u \Vert_{1, \mu(\mathcal{T})} - (s^k)^T u.
\end{align}
\item[\text{(S.2)}] If a suitable termination criterion, e.g., $u^k = u^{k+1}$, is satisfied: STOP.
\item[\text{(S.3)}] Set $k \leftarrow k+1$ and go to (S.1).
\end{itemize}
\end{algorithm}

Analogously to \cite{Ulbrich2014}, the necessary optimality condition of \eqref{DC-subproblem} can be formulated as
\begin{equation}\label{eq_F_tau}
F_\tau(u) := g(u) - \mathcal{P}_{[-\rho\mu(\triangle), \rho\mu(\triangle)]}\left(g(u) - \frac{1}{\tau} u \right) = 0 , \quad \tau > 0,
\end{equation}
where $g(u) := Au - b - s^k$ and $\mathcal{P}_{[-\rho\mu(\triangle), \rho\mu(\triangle)]}$ denotes the elementwise projection onto $[-\rho\mu(\triangle_i), \rho\mu(\triangle_i)]$.
For solving the DC subproblem \eqref{DC-subproblem} we use a local semismooth Newton method applied to the equation $F_\tau(u) = 0$.

We determine the subgradient $s \in \partial \vert w_u \vert_{K, \mu(\mathcal{T})}(u^k)$ in step (S.1) of \cref{DC-Algorithm} in the following way.

\begin{algorithm}\label{Alg:Subgradient}
\begin{itemize}
\item[]
\item[\text{(S.1)}] Determine the set $I^* \subseteq \{1,...,m\}$, which attains the maximum in $\vert w_{u^k} \vert_{K, \mu(\mathcal{T})}$.
\item[\text{(S.2)}] Define $r \in \R^m$ by $r_i := \mu(T_i)$ for $i \in I^*$ and $r_i := 0$ otherwise.
\item[\text{(S.3)}] Define $a \in \R^N$ by $a_i := \sign(u^k_i)$.
\item[\text{(S.4)}] Compute $s := (D^T r) \star a$, where $\star$ denotes elementwise multiplication.
\end{itemize}
\end{algorithm}

By \cref{lemmaSubdiff} and \cref{theoSubdiff}, the vector $s$ from step (S.4) then satisfies $s \in \partial \vert w_u \vert_{K, \mu(\mathcal{T})}(u^k)$.
As already mentioned, in our implementation the set $I^*$ is approximated by a greedy algorithm for knapsack problems, see \cite[Section 2.4]{Martello1990}.

\subsection{Numerical Experiments}\label{section:NumExperiments}

Let us report about numerical experiments for the solution of \eqref{Problem}.
For our experiments we choose $\Om := (0,1)^2$ and the measure space $(\Om, \mathcal{A}, \mu)$ to be the standard Lebesgue-space.
For the parameter function $g$ we use
\[
g(x,y) := 10x \sin(5x) \sin(7y).
\]
In accordance with \cref{section:DiscPiecewiseLin} we discretize and reformulate this problem by means of a standard triangular finite element method with continuous piecewise linear functions, which results in the penalized problem \eqref{PenProblem} to be solved.
All the numerical experiments were implemented in Python.
The mesh was generated using the package \textit{gmsh} \cite{gmsh} and the finite-element assembly was done using the \textit{scikit-fem} library \cite{skfem2020}.

If not mentioned otherwise, we initialize our DC-Algorithm, \cref{DC-Algorithm}, with $u^0 = A^{-1} b$, i.e., the solution of the unconstrained problem.
Linear systems of equations are solved by the function \textit{spsolve} from the python library \textit{scipy}.
Moreover, if not mentioned otherwise, we use
\[
K = 0.25, \quad \rho = 10^9, \quad h = \frac{1}{512},
\]
where $h$ is the target mesh size for the mesh-generator.
For the actual mesh sizes, i.e. $\max_{i \in \{1,...,m\}} \operatorname{diam}(T_i)$, and the degrees of freedom see \cref{Table:mesh_sizes}.

\begin{table}
\tbl{Actual mesh sizes and degrees of freedom.}
{\begin{tabular}{lccccccccc} \toprule
$1/h$     & 8 & 16 & 32 & 64 & 128 & 256 & 512 & 1024 & 2048 \\ \midrule
mesh size & $1.5 \cdot 10^{-1}$ & $8.3 \cdot 10^{-2}$ & $4.0 \cdot 10^{-2}$ & $1.9 \cdot 10^{-2}$ & $1.0 \cdot 10^{-2}$ & $5.1 \cdot 10^{-3}$ & $2.6 \cdot 10^{-3}$ & $1.3 \cdot 10^{-3}$ & $6.5 \cdot 10^{-4}$  \\
d.o.f.    & 66 & 276 & 1137 & 4631 & 18730 & 75355 & 302227 & 1209308 & 4840241 \\  \bottomrule
\end{tabular}}
\label{Table:mesh_sizes}
\end{table}

In general, the algorithm does not produce results which are exactly $0$ outside of what one would consider the support of the solution.
With respect to the $L^0$-pseudo-norm of a computed solution, we therefore interpret every value $\leq 10^{-10}$ as $0$.

We stop \cref{DC-Algorithm}  if $u^{k+1} = u^k$.
Surprisingly, this stopping criterion is always met in our numerical experiments.
To determine $\vert w_{u^k} \vert_{K,\mu(\mathcal{T})}$ as well as a subgradient $s$ in step (S.1) of \cref{DC-Algorithm}, a suitable index set for the weighted largest-$K$-norm, see \cref{Alg:Subgradient} (S.1), is determined by a greedy algorithm for knapsack problems, see \cite{Martello1990}.

The nonsmooth equation \eqref{eq_F_tau} is solved by a semismooth Newton method.
To incorporate the mesh-size as well as the size of the penalty parameter into the local subproblem we set
\[
\tau := \frac{100 \max_j(\vert u_j \vert)}{\rho \max_i(\mu(\triangle_i))}.
\]
The semismooth Newton method is stopped if $\Vert F_\tau(u) \Vert_2 \leq 10^{-14}$.

In the following we report about the influence of different parameters on the behavior of the algorithm.
In particular we compare the computed results with respect to functional value $f(u^*)$, the $L^0$-pseudo-norm $\Vert u^* \Vert_0 = \sum_{i=1}^m \mu(T_i) \vert \max(w_{u,i} - 10^{-10}, 0) \vert_0$, the feasability (feas.) $\Vert w_u \Vert_{1, \mu(\mathcal{T})} - \vert w_u \vert_{K, \mu(\mathcal{T})}$ and the total number of outer DC iterations (DC) and inner semismooth Newton iterations (ssN).

\subsubsection{Dependence on Level of Discretization}\label{Num_section:Disc}

\begin{table}
\tbl{Numerical results for \cref{Num_section:Disc}.}
{\begin{tabular}{lccccc} \toprule
$1/h$ & $f(u_h^*)$ & $\Vert u_h^* \Vert_0$ & feas. & DC & ssN \\ \midrule
8    & -0.0020 & 0.161 & $1.3 \cdot 10^{-17}$ & 4 & 4 \\
16   & -0.0058 & 0.235 & $1.2 \cdot 10^{-17}$ & 4 & 4  \\
32   & -0.0075 & 0.247 & $1.0 \cdot 10^{-17}$ & 4 & 4 \\
64   & -0.0083 & 0.249 & $-1.6 \cdot 10^{-17}$ & 4 & 4  \\
128  & -0.0087 & 0.249 & $-3.5 \cdot 10^{-17}$ & 4 & 4  \\
256  & -0.0088 & 0.250 & $-3.5 \cdot 10^{-18}$ & 5 & 5  \\
512  & -0.0089 & 0.250 & $-9.7 \cdot 10^{-17}$ & 5 & 5  \\
1024 & -0.0089 & 0.250 & $-1.4 \cdot 10^{-16}$ & 5 & 5 \\
2048 & -0.0089 & 0.250 & $1.2 \cdot 10^{-15}$ & 6 & 7 \\ \bottomrule
\end{tabular}}
\label{Table:Disc}
\end{table}

\begin{figure}
\centering
\subfloat[computed solution]{
\resizebox*{6cm}{!}{\includegraphics{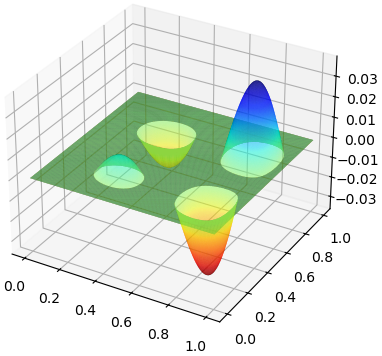}}}\hspace{5pt}
\subfloat[computed multiplier]{
\resizebox*{6cm}{!}{\includegraphics{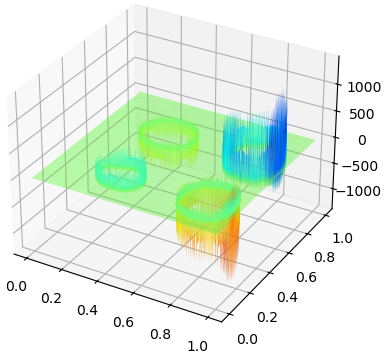}}}
\caption{Computed solution and multiplier for $1/h = 512$.}
\label{Fig:Solution512}
\end{figure}

The numerical results for different levels of discretization, i.e., different target mesh sizes $h$, can be seen in \cref{Table:Disc}.
The functional value $f(u_h^*)$ seems to converge for $h \searrow 0$, while the computed solution $u_h^*$ always satisfies the original constraint $\Vert u_h^* \Vert_0 \leq K = 0.25$.
The total number of inner semismooth Newton (ssN) and outer (DC) iterations only increases slightly across different values of $h$ and is overall small.

\subsubsection{Dependence on Penalty Parameter}\label{Num_section:Pen}

\begin{table}
\tbl{Numerical results for \cref{Num_section:Pen}.}
{\begin{tabular}{lccccc} \toprule
$\rho$    & $f(u_\rho^*)$ & $\Vert u_\rho^* \Vert_0$ & feas. & DC & ssN \\ \midrule
$10^3$    & -0.0089 & 0.250 & $-4.0 \cdot 10^{-17}$ & 5 & 6 \\
$10^4$    & -0.0089 & 0.250 & $3.1 \cdot 10^{-17}$ & 5 & 4  \\
$10^5$    & -0.0089 & 0.250 & $1.9 \cdot 10^{-16}$ & 4 & 3 \\
$10^6$    & -0.0089 & 0.250 & $1.7 \cdot 10^{-17}$ & 4 & 3  \\
$10^7$    & -0.0089 & 0.250 & $1.2 \cdot 10^{-16}$ & 5 & 5  \\
$10^8$    & -0.0089 & 0.250 & $-2.7 \cdot 10^{-16}$ & 5 & 5  \\
$10^9$    & -0.0089 & 0.250 & $-9.7 \cdot 10^{-17}$ & 5 & 5  \\
$10^{10}$ & -0.0089 & 0.250 & $-6.9 \cdot 10^{-18}$ & 5 & 5  \\
$10^{11}$ & -0.0089 & 0.250 & $-3.3 \cdot 10^{-17}$ & 5 & 5  \\
$10^{12}$ & -0.0088 & 0.250 & $-8.0 \cdot 10^{-17}$ & 5 & 5  \\
$10^{13}$ & -0.0089 & 0.250 & $5.7 \cdot 10^{-17}$ & 5 & 5  \\
$10^{14}$ & -0.0089 & 0.250 & $-2.9 \cdot 10^{-17}$ & 5 & 5  \\
$10^{15}$ & -0.0088 & 0.250 & $-1.7 \cdot 10^{-17}$ & 4 & 4  \\
$10^{16}$ & -0.0089 & 0.250 & $-2.4 \cdot 10^{-16}$ & 5 & 5  \\
$10^{17}$ & -0.0089 & 0.250 & $-1.9 \cdot 10^{-16}$ & 5 & 5  \\
$10^{18}$ & -0.0088 & 0.249 & $-1.2 \cdot 10^{-16}$ & 4 & 4  \\
$10^{19}$ & -0.0088 & 0.249 & $5.7 \cdot 10^{-17}$ & 4 & 4  \\ \bottomrule
\end{tabular}}
\label{Table:Pen}
\end{table}

The numerical results for different values of the penalty parameter $\rho$ can be seen in \cref{Table:Pen}.
The penalty parameter $\rho$ does not seem to have a large influence on the quality of the computed solution.
For $\rho = 10^{20}$ the python solver \textit{spsolve}, which is used to solve the Newton equation, breaks.

\subsubsection{Optimality Conditions}

Even though $\langle f^\prime(u^*), u^* \rangle = 0$ is not exactly satisfied, cf. \cref{Cor:Critical}, we have that $\langle f^\prime(u^*), u^* \rangle = (A u^* - b)^T u^* \approx 10^{-16}$.
Moreover, the multiplier, see \cref{Fig:Solution512}, seems to satisfy the necessary optimality conditions from \cref{cor_optimality_cond} or \cref{cor_optimality_cond_atomfree} respectively.

\subsubsection{Schedule for $K$} \label{Num_section:Sched}

\begin{table}
\tbl{Numerical results for \cref{Num_section:Sched}.}
{\begin{tabular}{llcccccc} \toprule
$1/h$ & $\lambda$ & $f(u_\lambda^*)$ & $\Vert u_\lambda^* \Vert_0$ & feas. & DC & ssN & $K$ \\ \midrule
128   & $0.90$    & -0.0072 & 0.250 & $-1.7 \cdot 10^{-18}$ & 16 & 28 & 14 \\
128   & $0.91$    & -0.0069 & 0.248 & $8.7 \cdot 10^{-18}$ & 18 & 30 & 15  \\
128   & $0.92$    & -0.0071 & 0.248 & $4.0 \cdot 10^{-17}$ & 20 & 31 & 17  \\
128   & $0.93$    & -0.0085 & 0.249 & $-2.4 \cdot 10^{-17}$ & 22 & 30 & 20 \\
128   & $0.94$    & -0.0111 & 0.249 & $-6.8 \cdot 10^{-17}$ & 26 & 35 & 23  \\
128   & $0.95$    & -0.0140 & 0.249 & $6.4 \cdot 10^{-17}$ & 30 & 37 & 28  \\
128   & $0.96$    & -0.0168 & 0.249 & $6.6 \cdot 10^{-17}$ & 37 & 43 & 34  \\
128   & $0.97$    & -0.0185 & 0.250 & $-1.4 \cdot 10^{-16}$ & 48 & 58 & 46  \\
128   & $0.98$    & -0.0192 & 0.250 & $-2.4 \cdot 10^{-17}$ & 71 & 89 & 69  \\
128   & $0.99$    & -0.0193 & 0.250 & $-1.2 \cdot 10^{-16}$ & 140 & 187 & 138 \\ \midrule
256   & $0.90$    & -0.0080 & 0.249 & $6.6 \cdot 10^{-17}$ & 17 & 29 & 14 \\
256   & $0.91$    & -0.0076 & 0.249 & $4.3 \cdot 10^{-17}$ & 18 & 31 & 15  \\
256   & $0.92$    & -0.0076 & 0.249 & $7.3 \cdot 10^{-17}$ & 20 & 35 & 17  \\
256   & $0.93$    & -0.0081 & 0.250 & $-8.7 \cdot 10^{-18}$ & 22 & 40 & 20 \\
256   & $0.94$    & -0.0081 & 0.249 & $-2.3 \cdot 10^{-17}$ & 26 & 41 & 23  \\
256   & $0.95$    & -0.0072 & 0.249 & $5.9 \cdot 10^{-17}$ & 31 & 45 & 28  \\
256   & $0.96$    & -0.0070 & 0.249 & $3.1 \cdot 10^{-17}$ & 37 & 49 & 34  \\
256   & $0.97$    & -0.0112 & 0.249 & $2.4 \cdot 10^{-17}$ & 48 & 55 & 46  \\
256   & $0.98$    & -0.0168 & 0.250 & $1.8 \cdot 10^{-16}$ & 72 & 80 & 69  \\
256   & $0.99$    & -0.0193 & 0.250 & $-6.9 \cdot 10^{-18}$ & 141 & 189 & 138 \\
256   & $0.995$   & -0.0193 & 0.250 & $1.5 \cdot 10^{-16}$ & 279 & 374 & 277 \\ \midrule
512   & $0.90$    & -0.0086 & 0.250 & $-2.3 \cdot 10^{-17}$ & 16 & 28 & 14 \\
512   & $0.91$    & -0.0085 & 0.250 & $-1.6 \cdot 10^{-17}$ & 18 & 31 & 15  \\
512   & $0.95$    & -0.0079 & 0.250 & $1.4 \cdot 10^{-16}$ & 31 & 46 & 28  \\
512   & $0.97$    & -0.0081 & 0.250 & $-9.9 \cdot 10^{-17}$ & 49 & 67 & 46  \\
512   & $0.98$    & -0.0073 & 0.250 & $2.0 \cdot 10^{-16}$ & 72 & 91 & 69  \\
512   & $0.99$    & -0.0169 & 0.250 & $2.9 \cdot 10^{-16}$ & 141 & 171 & 138 \\
512   & $0.995$   & -0.0193 & 0.250 & $1.8 \cdot 10^{-16}$ & 279 & 385 & 277 \\ \midrule
1024  & $0.98$    & -0.0076 & 0.250 & $-5.6 \cdot 10^{-17}$ & 72 & 116 & 69  \\
1024  & $0.99$    & -0.0080 & 0.250 & $-2.9 \cdot 10^{-16}$ & 141 & 189 & 138 \\
1024  & $0.995$   & -0.0159 & 0.250 & $2.5 \cdot 10^{-16}$ & 280 & 439 & 277 \\ \bottomrule
\end{tabular}}
\label{Table:Sched}
\end{table}

\begin{figure}
\centering
\subfloat[$\lambda=0.99$]{
\resizebox*{6cm}{!}{\includegraphics{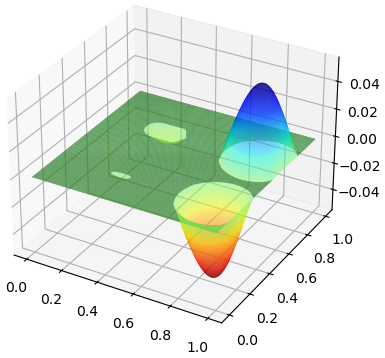}}\label{Fig:Schedule512a}}\hspace{5pt}
\subfloat[$\lambda=0.995$]{
\resizebox*{6cm}{!}{\includegraphics{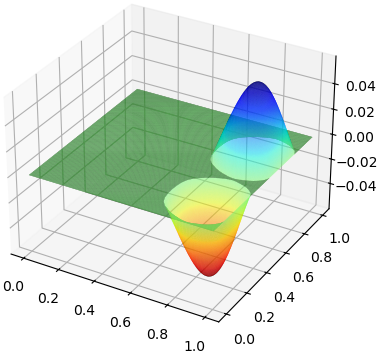}}\label{Fig:Schedule512b}}
\caption{Computed solutions with schedule for $1/h = 512$.}
\label{Fig:Schedule512}
\end{figure}

In \cite{Gotoh2017} they also tested a schedule for the constraint parameter $K$, which produced better results.
This means that, instead of using a fixed $K$ during the algorithm, one starts with $K = \mu(\Om)$ and decreases $K$ at every iteration until the target parameter is met.

Let $K \in (0, \mu(\Om))$ be the actual constraint parameter and $\lambda \in (0,1)$ a scheduling parameter.
Then in each iteration $k$ we replace the actual $K$ by
\[
K_k := \max\left( \lambda K_{k-1}, K \right),
\]
where $K_{-1} := \mu(\Om)$.
Obviously, in the implementation the stopping criterion $u^{k+1} = u^k$ is only active as soon as $K_k = K$.

In \cref{Table:Sched} one can see the results for different scheduling parameters $\lambda$ and different levels of discretization.
The last column indicates how many iterations are needed until the target parameter $K$ is met, i.e., $K_k = K$.
It shows that the number of iterations needed by the algorithm is (up to three iterations) equal to the iterations needed by the schedule.

In \cref{Fig:Schedule512} the computed solutions for two different scheduling parameters are depicted.
It indicates that the significantly better functional value for larger $\lambda$, and also compared to no schedule, cf. \cref{Fig:Solution512}, is due to a change in the support of the computed solution.

As shown by \cref{Table:Sched}, the size of the scheduling parameter has to be larger for a finer mesh to reach the better solution.
Moreover, at lower levels of discretization, we can observe some sort of convergence in the objective function values $f(u_\lambda^*)$ as $\lambda \nearrow 1$.
It can be expected that this can also be observed for higher levels of discretization, if $\lambda \nearrow 1$.

\subsubsection{Comparison with $L^1$-Penalization} \label{Num_section:L1}

The most famous method to induce sparsity is a $L^1$-penalization.
A possible relaxation for our original problem \eqref{Problem} would hence be
\[
\min_{u \in H_0^1(\Om)} \frac{1}{2} \Vert \nabla u \Vert_{L^2(\Om)}^2 - \int_{\Om} g u \dm + \beta \Vert u \Vert_{L^1(\Om)}, \qquad \beta > 0,
\]
with the approximate discretization
\begin{align}\label{ProblemL1}
\min_{u \in \R^N} \frac{1}{2} u^T A u - b^T u + \beta \Vert u_h \Vert_{1,h} =: \phi(u),
\end{align}
with $\Vert u_h \Vert_{1,h}$ from \cref{Def:Approx} and \cref{Theo:Approx}.
It is well known that for large enough $\beta$ the only solution of such $L^1$-penalized problems is zero, see, e.g., \cite{Kim2007}.

Problem \eqref{ProblemL1} can  be solved by using SpaRSA \cite{Sparsa2009}.
We use the simple non-monotone scheme with $u^0 = A^{-1} b$ and the following choice of parameters:
\[
M = 5, \quad \alpha_{\max} = 10^{20}, \quad \alpha_{\min} = \frac{1}{\alpha_{\max}}, \quad \eta = 2, \quad \sigma = 0.01.
\]
Furthermore, we stop the algorithm if
\[
\frac{\vert \phi(u^k) - \phi(u^{k-1}) \vert}{\vert \phi(u^{k-1}) \vert} \leq 10^{-5} \quad \text{and} \quad \frac{\Vert u^k - u^{k-1} \Vert_2}{\Vert u^k \Vert_2} \leq 10^{-5}.
\]
Note that the penalty parameter $\beta$ has to be tuned to reach a desired level of sparsity.

The results for different levels of sparsity can be seen in \cref{Table:L1}, where \textit{iter.} is the number of iterations needed and $\beta$ is the tuned penalty parameter.
The computed solution for $1/h = 512$ is depicted in \cref{Fig:SolutionL1}.
Note that the computed solution for the $L^1$-penalization is smoother compared to our approach, where there seems to be a nonsmoothness at the border of the support.
With respect to the objective function value, SpaRSA produces better results than our plain algorithm in lower levels of discretization.
However, if we use the scheduling from \cref{Num_section:Sched}, our algorithm clearly outperforms the $L^1$-approach.
Besides that, no tuning of the penalty parameter has to be done.
Moreover, it seems that the $L^1$-approach cannot handle larger levels of discretization that well since the objective function value increases again for $h \searrow 0$.
Nevertheless, with respect to the support, the computed solutions look similar to our solutions with scheduling, see \cref{Fig:SolutionL1} and \cref{Fig:Schedule512b}.

\begin{table}
\tbl{Numerical results for \cref{Num_section:L1}.}
{\begin{tabular}{lccccc} \toprule
$1/h$ & $f(u_h^*)$ & $\Vert u_h^* \Vert_0$ & feas. & iter. & $\beta$ \\ \midrule
16   & -0.0078 & 0.246 & $-3.5 \cdot 10^{-18}$ & 28 & 4.360   \\
32   & -0.0088 & 0.250 & $3.5 \cdot 10^{-18}$ & 50 & 4.260  \\
64   & -0.0091 & 0.250 & $-1.7 \cdot 10^{-18}$ & 82 & 4.215   \\
128  & -0.0092 & 0.250 & $3.5 \cdot 10^{-18}$ & 156 & 4.190   \\
256  & -0.0092 & 0.250 & $0$ & 466 & 4.170   \\
512  & -0.0091 & 0.250 & $-4.3 \cdot 10^{-17}$ & 772 & 4.165   \\
1024 & -0.0089 & 0.250 & $3.9 \cdot 10^{-16}$ & 1132 & 4.125   \\ \bottomrule
\end{tabular}}
\label{Table:L1}
\end{table}

\begin{figure}
\centering
\subfloat[Computed solution $u_{L^1}^*$ for problem \eqref{ProblemL1} with SpaRSA \cite{Sparsa2009}.]{
\resizebox*{6cm}{!}{\includegraphics{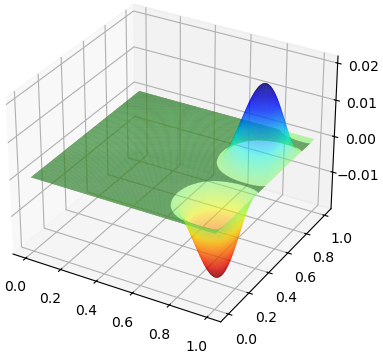}}
\label{Fig:SolutionL1}}\hspace{5pt}
\subfloat[Computed solution for $u^0 = u^*_{L^1}$.]{
\resizebox*{6cm}{!}{\includegraphics{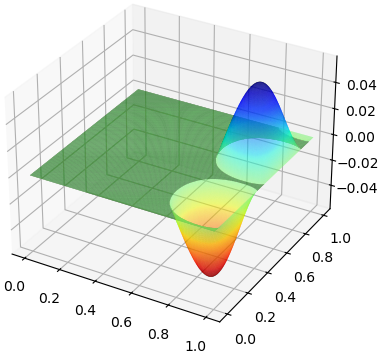}}
\label{Fig:SolutionInitL1}}
\caption{}
\end{figure}

\subsubsection{Dependence on Initialization}\label{Num_section:Init}

\begin{figure}
\centering
\subfloat[$a_i = 1: f(u^*) = -0.0134$]{
\resizebox*{4cm}{!}{\includegraphics{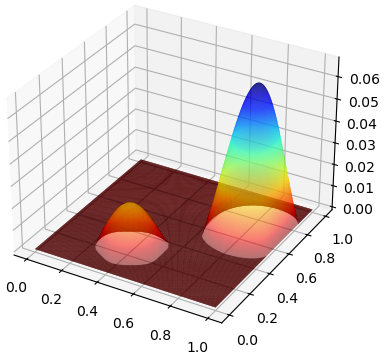}}}\hspace{5pt}
\subfloat[$a_i = -1: f(u^*) = -0.0131$]{
\resizebox*{4cm}{!}{\includegraphics{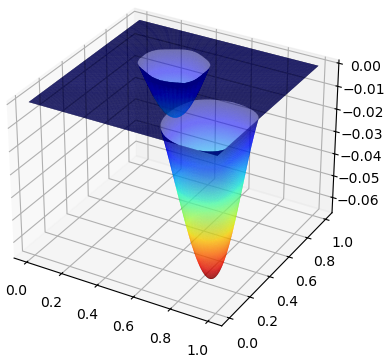}}}\hspace{5pt}
\subfloat[$a_i = \sign(b_i): f(u^*) = -0.0193$]{
\resizebox*{4cm}{!}{\includegraphics{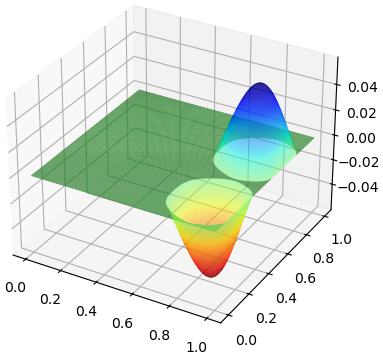}}}
\caption{Computed solutions for $u^0 = 0$, $\lambda=0.995$ and different subgradient selections with their respective functional values.}
\label{Fig:SolutionInit0}
\end{figure}

Another generic starting point is $u^0 = 0$.
However, this does not work without some changes.
This is due to the fact that we usually choose $s_i^k = 0$ if $u_i = 0$, see \cref{Alg:Subgradient}, and hence the DC subproblem degenerates to a $L^1$-penalized problem with zero solution, since $\rho$ is usually large in the sense of \cref{Num_section:L1}.
This problem can be overcome by choosing the subgradient $s^k$ in another way.

For $u^0 = 0$ we always apply the schedule from \cref{Num_section:Sched} with $\lambda = 0.995$ and stop the semismooth Newton method if $\Vert F_\tau(u) \Vert_2 \leq 10^{-10}$.
It turns out that the subgradient selection has a huge impact in the case $u^0 = 0$.
Natural choices for $a_i$ with $u_i^k=0$ in step (S.3) of \cref{Alg:Subgradient} would be $a_i = 1$ and $a_i = -1$.
A more problem specific choice is for example $a_i = \sign(b_i)$ since, in order to minimize the objective functional, $\sign(u_i) = \sign(b_i)$ seems to be advantageous.
The computed solutions for these choices of $a_i$ are depicted in \cref{Fig:SolutionInit0}.
The choice $a_i = \sign(b_i)$ seems to work very well since it reproduces the best functional value so far.

Furthermore, the non-negative (non-positive) computed solutions for $u^0 = 0$ and $a_i = 1$ ($a_i = -1$), see \cref{Fig:SolutionInit0}, are reproduced independently of the subgradient selection when using $u^0 = 1$ ($u^0 = -1$).

We moreover tested using the solution $u^*_{L^1}$ of the $L^1$-penalized problem from \cref{Num_section:L1} as a starting vector to our plain algorithm without scheduling.
This produced the solution from \cref{Fig:SolutionInitL1} and \cref{Table:InitL1}.
With respect to objective value the solution is not quite as good as the one starting at $u^0 = A^{-1} b$ and having a schedule with $\lambda=0.995$.
Nevertheless, this shows that our algorithm is first capable of and second needs very few iterations to significantly improve the solution of the $L^1$-problem.

\begin{table}
\tbl{Numerical results for \cref{Num_section:Init} with $u^0 = u^*_{L^1}$.}
{\begin{tabular}{lccccc} \toprule
$u$ & $f(u)$ & $\Vert u \Vert_0$ & feas. & DC & ssN \\ \midrule
$u^0$   & -0.0091    & 0.250  & $-4.3 \cdot 10^{-17}$ & -- & -- \\
$u^*$   & -0.0186    & 0.250  & $2.4 \cdot 10^{-16}$ & 2 & 2 \\ \bottomrule
\end{tabular}}
\label{Table:InitL1}
\end{table}

\section{An Optimal Control Problem}\label{Section:ControlProblem}

As a proof of concept, we report about numerical results to solve the following sparse optimal control problem:
Minimize
\begin{equation}\label{OptControlProblem}
\frac{1}{2} \Vert y - y_d \Vert_{L^2(\Om)}^2 + \frac{\alpha}{2} \Vert u \Vert_{L^2(\Om)}^2 + \frac{\beta}{2} \Vert \nabla u \Vert_{L^2(\Om)}^2
\end{equation}
over all $(u,y) \in H_0^1(\Om) \times H_0^1(\Om) $
satisfying the constraints
\begin{align*}
- \Delta y &= u \quad \text{ on } \Om, \nonumber\\
y &= 0 \quad \text{ on } \partial\Om, \nonumber\\
\Vert u \Vert_0 &\leq K.
\end{align*}
Here, $y_d \in L^2(\Om)$ is the desired state , $\alpha, \beta > 0$ and $K \in (0, \mu(\Om))$ are given parameters.
We use the reduced formulation
\begin{align*}
\min_{u \in H_0^1(\Om)} \frac{1}{2} \Vert Su - y_d \Vert_{L^2(\Om)}^2 + \frac{\alpha}{2} \Vert u \Vert_{L^2(\Om)}^2 + \frac{\beta}{2} \Vert \nabla u \Vert_{L^2(\Om)}^2 \quad \text{subject to} \quad \Vert u \Vert_0 \leq K,
\end{align*}
with the linear and bounded operator $S:=(-\Delta)^{-1} : L^2(\Omega) \to H_0^1(\Om)$. This problem fits into the setting \eqref{MinProblem}.
Moreover, all of our assumptions are satisfied.
Inspired by \cite{Stadler2009}, we set
\[
y_d = \frac{1}{6} \sin(2\pi x) \sin(2\pi y) \exp(2x)
\]
and $\alpha = 10^{-7}$. If not mentioned otherwise, we use $\beta = \alpha$.
In order to solve this problem numerically, we use our framework from \cref{Section:ExampleProblem} including the parameter and algorithmic setup from \cref{section:NumExperiments} with the default mesh size $h = 1/256$.
If we use the scheduling from \cref{Num_section:Sched}, we set $\lambda = 0.99$.
Note that the numerical example \cite[Example 2]{Stadler2009} can be interpreted as the $L^1$-relaxation of problem \eqref{OptControlProblem} with $\beta = 0$ (in the space $L^2(\Om)$).

The number of iterations needed by our algorithm without the scheduling from \cref{Num_section:Sched} is presented in \cref{Table:StadlerIter}.
The qualitative behavior of the solutions for $\beta \searrow 0$ is shown in \cref{Fig:controls_beta}, \cref{Fig:states_beta}, \cref{Fig:controls_beta_sched} and \cref{Fig:states_beta_sched}.
As can be seen, the support of $u$ has a finer structure with vanishing $\beta$, in addition, solutions for smaller $\beta$ appear to be less regular.
However, the corresponding state obviously approximates the desired state $y_d$ with respect to the $L^2(\Om)$-norm better for smaller $\beta$, see \cref{Table:beta}.
Moreover, the scheduling from \cref{Num_section:Sched} improves the computed solution.

\begin{table}
\tbl{Number of DC/ssN iterations needed to solve \eqref{OptControlProblem} without the scheduling from \cref{Num_section:Sched} for different values of $K$ and different levels of discretization.}
{\begin{tabular}{lcccc} \toprule
$K$           & $0.5$ & $0.25$ & $0.1$ & $0.01$ \\ \midrule
$h = 1/32$    & 3/3   & 3/3    & 3/3   & 3/3       \\
$h = 1/64$    & 3/3   & 4/4    & 3/3   & 5/5       \\
$h = 1/128$   & 3/3   & 3/3    & 3/3   & 4/4       \\
$h = 1/256$   & 3/3   & 3/3    & 4/4   & 4/4       \\
$h = 1/512$   & 4/4   & 3/3    & 3/3   & 4/4       \\ \bottomrule
\end{tabular}}
\label{Table:StadlerIter}
\end{table}

\begin{table}
\tbl{Values of $\Vert Su - y_d \Vert_{L^2(\Om)}$ without and with the scheduling from \cref{Num_section:Sched} for $K = 0.25$ and different values of $\beta$.}
{\begin{tabular}{lccccccc} \toprule
$\beta$  & $10^{-7}$            & $10^{-8}$            & $10^{-9}$            & $10^{-10}$           & $10^{-11}$              & $10^{-12}$           & $10^{-13}$     \\ \midrule
without  & $1.14 \cdot 10^{-1}$ & $5.78 \cdot 10^{-2}$ & $3.91 \cdot 10^{-2}$ & $3.69 \cdot 10^{-2}$ & $3.52 \cdot 10^{-2}$ & $3.38 \cdot 10^{-2}$ & $3.31 \cdot 10^{-2}$ \\
with     & $7.84 \cdot 10^{-2}$ & $3.70 \cdot 10^{-2}$ & $2.64 \cdot 10^{-2}$ & $1.54 \cdot 10^{-2}$ & $8.15 \cdot 10^{-3}$ & $4.56 \cdot 10^{-3}$ & $3.02 \cdot 10^{-3}$ \\ \bottomrule
\end{tabular}}
\label{Table:beta}
\end{table}

\begin{figure}
\centering
\subfloat[$\beta = 10^{-7}$]{
\resizebox*{4cm}{!}{\includegraphics{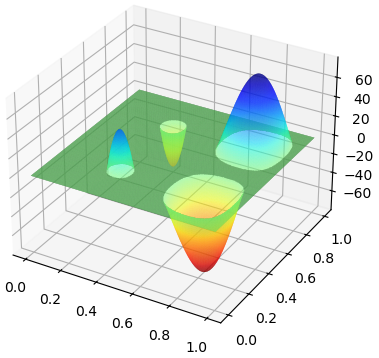}}}\hspace{5pt}
\subfloat[$\beta = 10^{-10}$]{
\resizebox*{4cm}{!}{\includegraphics{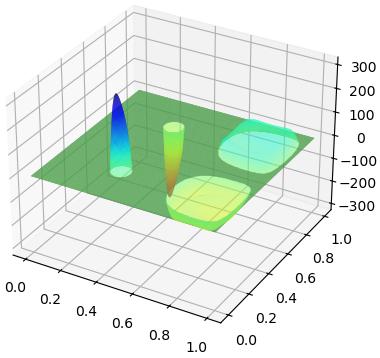}}}\hspace{5pt}
\subfloat[$\beta = 10^{-13}$]{
\resizebox*{4cm}{!}{\includegraphics{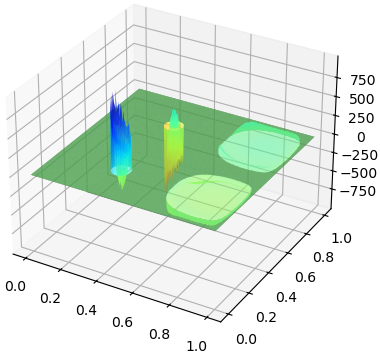}}}
\caption{Computed controls for $K = 0.25$ without the scheduling from \cref{Num_section:Sched} for different values of $\beta$.}
\label{Fig:controls_beta}
\subfloat[$\beta = 10^{-7}$]{
\resizebox*{4cm}{!}{\includegraphics{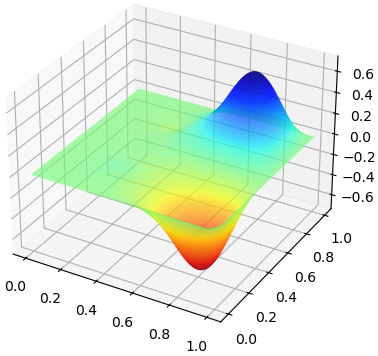}}}\hspace{5pt}
\subfloat[$\beta = 10^{-10}$]{
\resizebox*{4cm}{!}{\includegraphics{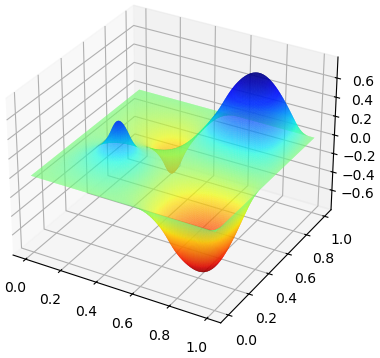}}}\hspace{5pt}
\subfloat[$\beta = 10^{-13}$]{
\resizebox*{4cm}{!}{\includegraphics{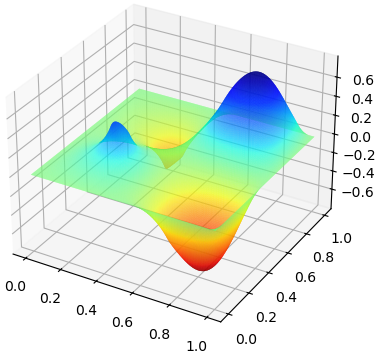}}}
\caption{Computed states for $K = 0.25$ without the scheduling from \cref{Num_section:Sched} for different values of $\beta$.}
\label{Fig:states_beta}
\subfloat[$\beta = 10^{-7}$]{
\resizebox*{4cm}{!}{\includegraphics{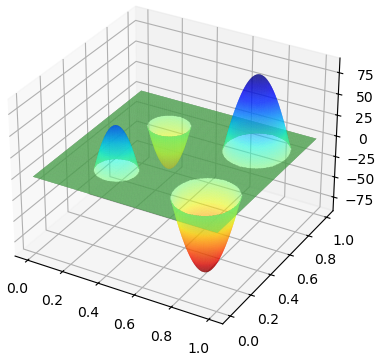}}}\hspace{5pt}
\subfloat[$\beta = 10^{-10}$]{
\resizebox*{4cm}{!}{\includegraphics{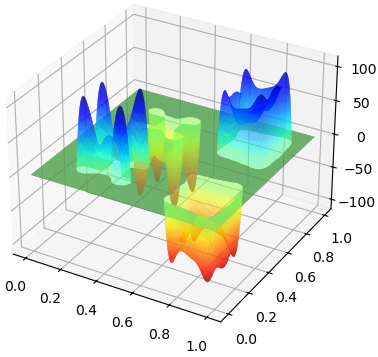}}
}\hspace{5pt}
\subfloat[$\beta = 10^{-13}$]{
\resizebox*{4cm}{!}{\includegraphics{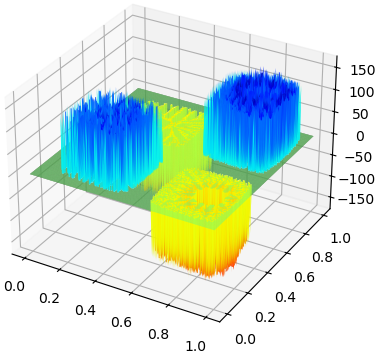}}}
\caption{Computed controls for $K = 0.25$ with the scheduling from \cref{Num_section:Sched} for different values of $\beta$.}
\label{Fig:controls_beta_sched}
\subfloat[$\beta = 10^{-7}$]{
\resizebox*{4cm}{!}{\includegraphics{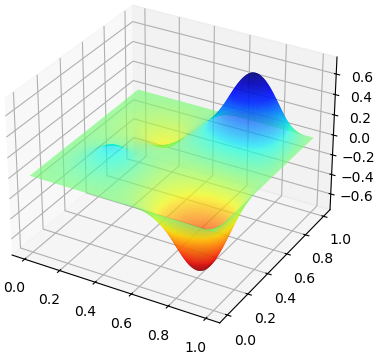}}}\hspace{5pt}
\subfloat[$\beta = 10^{-10}$]{
\resizebox*{4cm}{!}{\includegraphics{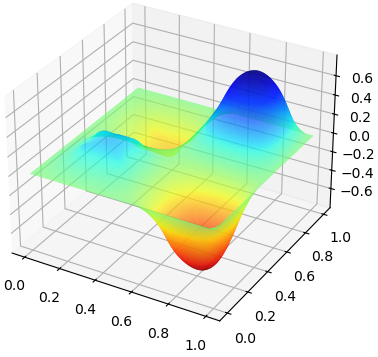}}}\hspace{5pt}
\subfloat[$\beta = 10^{-13}$]{
\resizebox*{4cm}{!}{\includegraphics{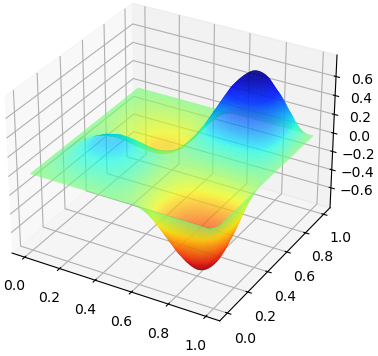}}}
\caption{Computed states for $K = 0.25$ with the scheduling from \cref{Num_section:Sched} for different values of $\beta$.}
\label{Fig:states_beta_sched}
\end{figure}

\section*{Funding}

This research was partially supported by the German Research Foundation DFG under project grant Wa 3626/3-2.

\bibliographystyle{tfs}
\bibliography{Literatur}

\end{document}